\documentclass[12pt,letterpaper,reqno]{amsart}

\usepackage{amssymb,amsmath}
\usepackage{amsthm}
\usepackage{color,graphicx}
\usepackage{eucal}

\usepackage{amsfonts, amstext, amsthm, latexsym}

\newcommand{\N}{{\mathbb N}}

\newcommand{\RR}{{\mathbb R}}

\numberwithin{equation}{section}
\newtheorem{theorem}{Theorem}[section]

\newtheorem{lemma}[theorem]{Lemma}
\newtheorem{remark}[theorem]{Remark}

\newtheorem{Prop}[theorem]{Proposition}

\begin{document}

\title[Stochastic gKdV equation]{Generalized KdV equation\\ 
subject to a stochastic perturbation}

\author[A. Millet]{ Annie Millet}
\address{SAMM, EA 4543,
Universit\'e Paris 1 Panth\'eon Sorbonne, 
{\it and} Laboratoire de
Probabilit\'es et Mod\`eles Al\'eatoires,
Universit\'es Paris~6-Paris~7, 
Paris, France}
\email{annie.millet@univ-paris1.fr {\it and}
annie.millet@upmc.fr}

\author[S. Roudenko]{Svetlana Roudenko}
\address{The George Washington University, Department of Mathematics, 
Washington DC, USA}
\email{roudenko@gwu.edu}


\begin{abstract}
We prove global well-posedness of the subcritical generalized Korteweg-de Vries equation (the mKdV and the gKdV with quartic power of nonlinearity) subject to an additive random
perturbation. More precisely, we prove that if the driving noise is a cylindrical Wiener process on $L^2(\RR)$ and the covariance operator is Hilbert-Schmidt in an appropriate Sobolev space, then the solutions with $H^1(\RR)$ data are globally well-posed in $H^1(\RR)$. This extends results obtained by A. de Bouard and A. Debussche for the stochastic KdV equation.
\smallskip

\hfill{\it Dedication: In the memory of Igor Chueshov.}
\end{abstract}



\maketitle


\section{Introduction}\label{s1}
In this paper we study a subcritical generalization of the Korteweg-de Vries (gKdV) equation subject to some additive random perturbation $f(t)$, that is,
\begin{equation}
\label{gKdV}
\partial_t u(t) + \partial^3_x u(t) +  \mu \, u(t)^k \partial_x u(t) = f(t), \;~~ (x, t) \in \RR \times \RR, \quad u(0,.)=u_0,
\end{equation}
with $k=2$, the mKdV case, or $k=3$, referred to as the gKdV equation.  
Here, $\mu = \pm 1$, which is referred to as focusing or defocusing nonlinearity.

The well-known KdV equation ($k=1$) describes the propagation of long waves in a channel.
Its generalizations ($k>1$) appear in several physical systems; a large class of hyperbolic models can be reduced to these equations.
The well-posedness in the KdV equation has been extensively studied by many authors in the deterministic setting without any forcing term
($f=0$) and goes back to works of Kato \cite{Kato_1983}, Kenig-Ponce-Vega \cite{KPV_CPAM} to name a few;
there is an abundant literature available on that. The question about the minimal regularity assumptions on initial data needed for
well-posedness has been also investigated intensively in recent years; two important methods should be mentioned: the so-called I-method (e.g.,
\cite{CKSTT}) and the probabilistic approach of randomizing the initial data and showing the invariance of Gibbs measures (e.g., \cite{Bou}, \cite{Zh}).
In this paper we also take a probabilistic approach, however, in a completely different setting, where the equation itself has a random term.
We do not aim to obtain the lowest possible regularity for such an equation, but simply show how to combine the deterministic
and probabilistic approaches in this case to study well-posedness for the initial data with finite energy.

In \cite{KPV_CPAM}, Kenig, Ponce, Vega showed that for $k=1, 2, 3$, if $u_0\in H^1(\RR)$,
the subcritical gKdV equation has a global solution in $L^\infty\big([0,\infty) ; H^1(\RR)\big)$.
In the critical case $k=4$ (resp. supercritical case $k>4$), there is a local existence in $H^s(\RR)$ with  $s>0$
(resp. in $\dot{H}^{s_k}(\RR)$ for $s_k=(k-4)/(2k)$), when $u_0$ belongs to the corresponding Sobolev space. Global well-posedness
holds if the $L^2(\RR)$ norm of $u_0$ (resp. the $L^2(\RR)$-norm of $D^{s_k}u_0$) is small.

Here, we study the subcritical case of the generalized KdV equation,
\begin{equation}
\label{S-gKdV}
d u_t + \big( \partial^3_x u(t) +  \mu \, u(t)^k \partial_x u(t) \big) dt = d\,f(t) \equiv
\Phi dW(t),
\end{equation}
where an external random forcing $f$ is driven by a cylindrical Brownian motion $W$ on $L^2(\RR)$ and
multiplied by some smoothing covariance operator $\Phi$.
The driving Wiener process $W$ describes a noise in the environment, that is, a sum of little independent shocks properly renormalized.
The smoothing operator describes spatial correlation of the noise, but the time increments of $\Phi W$ are independent,
that is,  the noise is white in time.
The stochastic KdV equation ($k=1$) on $\RR$ has been studied in a series of papers by A.~de Bouard and A.~Debussche (see e.g.
\cite{deB_Deb_sKdV}, \cite{deB_Deb_Tsu},  \cite{deB_Deb_homogeneous}).  
In \cite{deB_Deb_sKdV} they proved that if $u_0\in H^1(\RR)$ and if $\Phi$ is a Hilbert-Schmidt operator from $L^2(\RR)$ to $H^1(\RR)$,
then there is a global solution to the stochastic KdV equation which belongs a.s. to $C([0,T]; H^1(\RR))$. Using Bourgain spaces,
when $u_0\in L^2(\RR)$ and the covariance operator $\Phi$ is Hilbert-Schmidt both from $L^2(\RR)$ to $L^2(\RR)$ and to $\dot{H}^{-3/8}(\RR)$,  
they have shown in \cite{deB_Deb_Tsu} the existence and uniqueness of the solution  in $L^2\big(\Omega ; C([0,T]; L^2(\RR))\big)$ for any $T>0$.
Note that for the mKdV or gKdV equations, the Bourgain  spaces  approach to lower the regularity of global solutions
is not needed (since it gives the same results but is more technically involved). Therefore, for mKdV and gKdV, $k\geq 2$, it suffices
to use arguments from \cite{KPV_CPAM}.
In \cite{deB_Deb_homogeneous}, the  authors  have proved the global well-posedness of solutions to the stochastic KdV equation in $L^2(\RR)$
(resp. $H^1(\RR)$), when the noise is homogeneous, that is, of the form $u(s) \phi dW(s)$ for a convolution operator $\phi$
defined in terms of an $L^2(\RR) \cap L^1(\RR)$ (resp. $H^1(\RR)\cap L^1(\RR)$) kernel.  They used the Bourgain space approach,
which is necessary to lower regularity of solutions in the KdV case; it is also helpful when dealing with multiplicative noise.

We do not give a full reference for the stochastic KdV and related equations in the periodic setting. However, to guide the reader in the proper direction, we mention a few results. T.~Oh \cite{Oh} studied a stochastic KdV equation on the torus ${\mathcal T}=[0, 2\pi)$. For specific assumptions on the covariance operator $\Phi$, he proved that there is a local well-posedness in a certain Bourgain space if the initial condition belongs to it as well. Other KdV-type models can also be considered with variations of the additive noise, such as adding a derivative to the noise (e.g., see the work of G.~Richards \cite{Richards}).

The main goal of this paper is to obtain the global well-posedness of solutions to the mKdV and gKdV with $k=3$ equations in $H^1(\RR)$; global solutions with finite energy are important for physical applications, i.e., in study of solitary waves. To study well-posedness, we need to set up a specific functional framework that provides the necessary flexibility to use smoothing properties of the Airy group while considering the stochastic term. We note that we consider a driving cylindrical Wiener process, which is quite usual in nonlinear dispersive hyperbolic models, such as the stochastic nonlinear
Schr\"odinger (NLS) equation, this, in its turn, requires the use of non-Hilbert Sobolev spaces. We now state the main result and refer the reader to the next section for all notations.

\begin{theorem}  \label{global}
Let $u_0$ be ${\mathcal G}_0$-measurable and belong to $H^1_x$ a.s.
\begin{enumerate}
\item
Let $k=2$ and $\Phi\in L_2^{0, 1+\epsilon}$ for some $\epsilon >0$. Then given any positive time $T$, there exists a unique solution
  to   \eqref{S-gKdV}  which belongs a.s. to ${\mathcal X}_2^T \cap C([0,T],H^1_x)$.
  Furthermore, if $u_0\in L^2_\omega(H^1_x) \cap   L^6_\omega(L^2_x)$, then $u\in L^2_\omega(L^\infty_t(H^1_x))$.

\item
  Let $k=3$ and $\Phi\in L_2^{0, 1}$. Then given any positive time $T$, there exists a unique solution   to  \eqref{S-gKdV}
  which belongs a.s. to ${\mathcal X}_3^T \cap C([0,T],H^1_x)$.
  Furthermore, if $u_0\in L^2_\omega(H^1_x) \cap   L^{14}_\omega(L^2_x)$, then $u\in L^2_\omega(L^\infty_t(H^1_x))$.
\end{enumerate} 
\end{theorem}

While we follow the main framework of \cite{deB_Deb_sKdV}, additional difficulties appear which are due to higher power of nonlinearity considered. 
When $k=2$,  $u_0\in H^{1/4}(\RR)$ and $\Phi$ is Hilbert-Schmidt from $L^2(\RR)$
to $H^{1+\epsilon}(\RR)$ for some $\epsilon>0$, we prove that there exists
a unique solution until some stopping time $T_2>0$. The hypothesis on $\Phi$ with some ``larger derivative"
is due to the  functional space $L^4_x(L^\infty_t)$.
A similar space $L^2_x(L^\infty_t)$ appears for the fixed point argument of the KdV equation; technical problems arise going from $L^2_x$ to $L^4_x$.
When  $k=3$,   $u_0\in H^{1/12}(\RR)$ and $\Phi$ is Hilbert-Schmidt from $L^2(\RR)$ to $H^{5/12}(\RR)$), we prove that there exists
a unique solution until some stopping time $T_3>0$.
This technique does not easily extend to multiplicative noise; indeed change of variables in time is  no longer possible for moments of
norm estimates of the corresponding stochastic integral. The problem of multiplicative noise will be addressed  elsewhere.

The paper is organized as follows.  In section \ref{SIntegral}, we prove some technical lemmas on functional properties of the stochastic integral
$\int_0^t S(t-s) \Phi\, dW(s)$. Using the functional framework introduced in \cite{KPV_CPAM} and a contraction principle
in an appropriate function space, we prove local well-posedness  of the solution in section \ref{Local}.
In section \ref{Global}, we prove that if the initial condition belongs to $H^1(\RR)$ and
$\Phi$ is Hilbert-Schmidt from $L^2(\RR)$ to $H^{1+\epsilon}(\RR)$ when $k=2$ (and  from $L^2(\RR)$  to $H^1(\RR)$ when $k=3$),
the solution can be extended to any given time interval $[0,T]$. Then  it belongs to 
$L^2(\Omega; L^\infty(0,T ; H^1(\RR)))$, and takes a.s. its values in the set of continuous
trajectories from $[0,T]$ to $H^1(\RR)$.
The proof uses the time invariance of   mass and  Hamiltonian for  solutions  to the deterministic gKdV equation.
In order to use these invariant quantities, we need a more regular solution.
This is achieved approximating the solution $u$ by a sequence 
$\{u_n\}_n$ of solutions defined in terms of  smoother initial conditions $u_{0,n}$ and of  more regularizing operators $\Phi_n$.

The first named author collaborated with Igor Chueshov on general 2D hydrodynamical models related with the Navier-Stokes equations.
In this paper, we try to further develop the intertwining between deterministic and stochastic approaches in PDEs. Such interplay was one
of the fundamental contributions of Igor Chueshov's scientific work.

\section{Local existence  of the solution}\label{SIntegral}
In this section we study the stochastic generalized KdV equation with additive noise defined for $x\in \RR$ and $t\geq 0$
\begin{equation}   \label{stoch_gKdV}
d u(t) + \big( \partial^3_x u(t) +  \mu \,   u(t)^k \partial_x u(t) \big) dt = \Phi dW(t), \quad k=2,3,
\end{equation}
with the initial condition $u(x,0)=u_0(x)$.  From now on we will assume $\mu =1$ (focusing case); the defocusing case follows automatically.
The case $k=1$, which is that of the stochastic KdV equation, has been studied in
\cite{deB_Deb_sKdV} and \cite{deB_Deb_Tsu}. Here, $W$ is a cylindrical Wiener process on $L^2(\RR)$ adapted to a filtration $({\mathcal G}_t, t\geq 0)$,
 that is,  $W(t) \varphi =\sum_{j\in {\mathbb N}}  (e_j,\varphi) \, \beta_j(t)$ for any $\varphi \in L^2(\RR)$,
 where the processes $\beta_j(t)$, $ j\geq 0$ are independent one-dimensional Brownian
motions  adapted to  $({\mathcal G}_t)$  and $\{e_j\}_{j\geq 0}$ is an orthonormal basis of $L^2(\RR)$, often referred to as a CONS (complete
orthonormal system).
Note that the process $W(t)$ is not $L^2(\RR)$-valued, but  $W(t)\varphi$ is a centered Gaussian random variable with variance $\|\varphi\|_{L^2}^2
= \sum_{j\geq 0} (e_j, \varphi)^2$.
We suppose that   $\Phi$ is a linear map which is Hilbert-Schmidt from $L^2$ into $H^\sigma(\RR)$ for some non-negative $\sigma$, that is,
\begin{equation}     \label{defPhi}
\| \Phi \|_{L^{0,\sigma}_2}:= \| \Phi\|_{L^0_2(L^2(\RR),H^\sigma(\RR))} <\infty.
\end{equation}
We suppose that $u_0$ is ${\mathcal G}_0$-measurable and  $H^1$-valued.

As in \cite{deB_Deb_sKdV} using Duhamel's formula
we write this equation using its
{mild} formulation, that is,
\begin{equation}     \label{mild_gKdV}
u(t)=S(t) u_0 - \int_0^t S(t-s) \big(u(s)^k \partial_x u(s)\big) ds+ \int_0^t S(t-s) \Phi dW(s),
\end{equation}
where
\[ S(t) u = {\mathcal F}_\xi^{-1}\big( e^{it\xi^3} \hat{u}(\xi) \big),\]
and ${\mathcal F}(u) = \hat{u}$ denotes the Fourier transform of $u$.
 Note that
 \[\int_0^t S(t-s) \Phi dW(s)=\sum_{j\geq 0} \int_0^t S(t-s) \Phi e_j d\beta_j(s)\]
  is  a centered  $H^\sigma(\RR)$ - valued Gaussian variable. Since $S(t-s)$ is an 
$H^\sigma(\RR)$ isometry for all $\sigma \geq 0$, the variance of this stochastic integral is
$$
\int_0^t \sum_{j\geq 0} \|\Phi e_j\|_{H^\sigma_x}^2 ds = t \| \Phi\|_{L^{0,\sigma}_2}^2.
$$
Following the approach in \cite{KPV_CPAM} (and \cite{deB_Deb_sKdV} for the case $k=1$), 
we introduce the following spaces of functions $u:\RR\times [0,T]\to \RR$:
\begin{eqnarray}
{\mathcal X}_2^T&=& \big\{ u\in C\big([0,T];H^{1/4}(\RR) \big) \cap L^4_x\big(L^\infty_t\big) :
 D_x u \in L^{20}_x  \big(L^{5/2}_t\big) , \;
 \nonumber \\
 &&  \;
D_x^{1/4} u \in L^{5}_x \big(L^{10}_t\big) , \;  D_x^{1/4} \partial_x u \in L^{\infty}_x \big(L^{2}_t\big) \big\}
\label{spacek=2}
\end{eqnarray}
for the mKdV equation ($k=2$), and
\begin{eqnarray}
{\mathcal X}_3^T&=&  \big\{ u\in C\big([0,T];H^{1/12}(\RR) \big) \cap L^{42/13}_x\big(L^{21/4}_t\big)\cap L^{60/13}_x\big( L^{15}_t\big) \cap
L^{10/3}_x \big( L^{30/7}_t\big)  :
 \nonumber \\
 &&  \;
D_x^{1/12} u \in L^{10/3}_x \big( L^{30/7}_t\big) , \;  \partial_x u \in L^{\infty}_x \big(L^{2}_t\big) , \;
D_x^{1/12} \partial_x u \in  L^{\infty}_x \big(L^{2}_t\big) \big\}
\label{spacek=3}
\end{eqnarray}
for the gKdV equation ($k=3$). Here, $L^q_x$ (resp. $L^p_t$) denotes $L^q(\RR)$ (resp. $L^p(0,T)$).

In order to prove that the process $v$, defined by the stochastic integral
\begin{equation}     \label{IS}
v(t):=\int_0^t S(t-s) \Phi dW(s), \quad t\in [0,T],
\end{equation}
belongs a.s. to the spaces ${\mathcal X}_k^T$ for $k=2, 3$ under proper assumptions on the operator $\Phi$, we  first prove some technical lemmas. In each result we state the minimal regularity assumption on the operator $\Phi$ and the corresponding  power of $T$
obtained in the upper estimate, in order to deal with the ${\mathcal X}_k^T$-norm of $v$. 
\smallskip

The following lemma is a  generalization of Proposition 3.1 in \cite{deB_Deb_sKdV}.
\begin{lemma}        \label{lemunifH}
Let $\sigma \geq 0$ and $q \in [1,\infty)$. Then for every $T>0$, we have
\[ E\Big( \sup_{t\in [0,T]} \|v(t)\|_{H^\sigma_x}^{2q}\Big) \leq C_q\,  T^{q} \, \| \Phi\|_{L^{0,\sigma}_2}^{2q}.
\]
\end{lemma}
\begin{proof}
The proof is quite classical; it is sketched for the sake of completeness.
The upper estimate  is proved for $q\in [2,\infty)$ and deduced for
$q\in [1,\infty)$ by H\"older's inequality.  Let $J_\sigma u  = {\mathcal F}^{-1}
\Big( (1+| \xi |^2)^{\frac{\sigma}{2}} \hat{u}(\xi)\Big)$. 
First, note that since $S(t)$ is a group and an $H^\sigma_x$-isometry, we have $\|v(t)\|_{H_x^\sigma} = \| \bar{v}(t)\|_{H_x^\sigma}$,
where $\bar{v}(t)=\int_0^t S(-s) \Phi dW(s)$.
For fixed $t\in [0,T]$ the random variable $\bar{v}(t)$ is   an  $H_x^\sigma$ - valued, Gaussian with mean zero and variance
$\int_0^t \sum_{j\geq 0} \|J_\sigma S(-s) \Phi e_j \|_{L^2_x}^2 \, ds = t\|\Phi\|_{L_2^{0,\sigma}}^2$,
where $\{e_j\}_{ j\geq 0}$ is the  CONS  of $L^2(\RR)$  in the definition of $W$.
It\^o's formula implies
\[ \|\bar{v}(t)\|_{H^\sigma_x}^2 = 2\int_0^t \big( J_\sigma \bar{v}(s) , J_\sigma  S(-s) \Phi dW(s) \big) + \int_0^t \sum_{j\geq 0}
\|J_\sigma S(-s) \Phi e_j \|_{L^2_x}^2  ds
\]
for every $t\in [0,T]$. Using once more the It\^o formula, we deduce that
for $q\in [ 2,\infty)$,  we have $\|\bar{v}(t)\|_{H^\sigma_x}^{2q} = \sum_{i=1}^3 T_i(t)$, where
\begin{align*}
 T_1(t)& = 2q \int_0^t \big( J_\sigma \bar{v}(s) , J_\sigma  S(-s) \Phi dW(s) \big) \, \|\bar{v}(s)\|_{H^\sigma_x}^{2(q-1)},\\
 T_2(t)&= q\int_0^t \| \Phi\|_{L^{0,\sigma}_2}^2 \, \|\bar{v}(s)\|_{H^\sigma_x}^{2(q-1)}\, ds, \\
 T_3(t)&=2q(q-1) \int_0^t  
 \sum_{j\in \N} \big(  J_\sigma  S(-s) \Phi e_j\, , \, J_\sigma \bar{v}(s)\big)^2 \, \|\bar{v}(s)\|_{H^\sigma_x}^{2(q-2)}\, ds.
 \end{align*}
The Cauchy-Schwarz inequality  implies
 \[ E \Big( \sup_{t\in [0,T]} (T_2(t) + T_3(t)) \Big) \leq C_q \|\Phi\|_{L^{0,\sigma}_2}^2 E\Big( \int_0^T \|\bar{v}(s)\|_{H^\sigma_x}^{2(q-1)} ds\Big)
 \leq C_q \, T^q\,  \|\Phi\|_{L^{0,\sigma}_2}^{2q}.
 \]
The Davies inequality for martingales, Young's inequality  and Fubini's theorem imply
 \begin{align*}
 E \Big(& \sup_{t\in [0,T]} T_1(t) \Big)  \leq 6q\, E\Big( \Big\{ \int_0^T \|  \bar{v}(s)\|_{H_x^\sigma}^{4(q-1)} \,
\sum_{j\geq 0} (J_\sigma \bar{v}(s), J_\sigma S(-s) \Phi e_j)^2  ds \Big\}^{\frac{1}{2}} \Big) \\
 &\leq 6q \, \sqrt{T}\, \, \|\Phi\|_{L^{0,\sigma}_2}\;   E\Big(  \sup_{s\in [0,T]} \|v(s)\|_{H^\sigma_x}^{2q-1} \Big) 
 \leq  \frac{1}{2} E\Big(  \sup_{s\in [0,T]} \|v(s)\|_{H^\sigma_x}^{2q} \Big) + C_q T^q \|\Phi\|_{L^{0,\sigma}_2}^{2q},
 \end{align*}
which concludes the proof.
 \end{proof}

The following result will be used to upper estimate one of the norms in the definition of $\|v \|_{ {\mathcal X}^T_2}$.
\begin{lemma}    \label{lem0}
Let $p,q$ satisfy $2\leq p\leq q<\infty$ and  $\sigma \geq 0$. Then for some $C>0$, we have
\begin{equation}      \label{infty_p}
\| D^{\sigma+1}_x v\|_{L^\infty_x(L^q_\omega(L^p_t))}\leq C \, T^{\frac{1}{p}} \, \|\Phi\|_{L_2^{0,\sigma}}.
\end{equation}
\end{lemma}
\begin{proof}
Since $q\geq p$, H\"older's inequality with respect to $dt$, Fubini's theorem and moments of the stochastic integral yield
for the CONS $\{e_j\}_{j\geq 0}$ of $L^2(\RR)$ in the definition of $W$:
\begin{align*}
\sup_{x\in \RR} E\Big| \int_0^T& \big| D^{\sigma+1}_x \int_0^t S(t-s) \Phi dW(s) \big|^p dt \Big|^{\frac{q}{p}} \\
& \leq T^{\frac{q}{p}-1} \; \sup_{x\in \RR} E\Big( \int_0^T \big| D^{\sigma+1}_x \int_0^t S(t-s) \Phi dW(s) \big|^q dt \Big) \\
 & \leq  C_q\, T^{\frac{q}{p}-1} \; \sup_{x\in \RR} \int_0^T \Big| \sum_{j\geq 0}
 \int_0^t | D^{\sigma+1}_x S(t-s) \Phi e_j\big|^2 ds \Big|^{\frac{q}{2}} dt \\
 &  \leq C_q\, T^{\frac{q}{p}-1} \;  \int_0^T  \sup_{x\in \RR} \Big| \sum_{j\geq 0}
 \int_0^t | D^{\sigma+1}_x S(t-s) \Phi e_j\big|^2 ds \Big|^{\frac{q}{2}} dt \\
&  \leq C_q\, T^{\frac{q}{p}-1} \;  \int_0^T  \Big| \sum_{j\geq 0}
 \; \sup_{x\in \RR} \int_0^t | D^{\sigma+1}_x S(t-s) \Phi e_j\big|^2 ds \Big|^{\frac{q}{2}} dt .
\end{align*}
The local smoothing property (see Lemma 2.1 in \cite{KPV_JAMS}) implies that for every $j\in \N$ and $t\in [0,T]$
\[ \sup_{x\in \RR} \int_0^t | D^{\sigma+1}_x S(s) \Phi e_j\big|^2 ds \leq C\, \|D^\sigma_x \Phi e_j\|_{L^2(\RR)}^2
\leq C \|\Phi e_j\|_{H^\sigma_x}^2.
\]
Therefore,
\[ \|D^{\sigma+1}_x v\|^q_{L^\infty_x(L^q_\omega(L^p_t))}
\leq C\,  T^{\frac{q}{p}-1 } \, T \, \Big( \sum_{j\geq 0} \|\Phi e_j \|_{H^\sigma_x}^2\Big)^{\frac{q}{2}}
 \leq C \, T^{\frac{q}{p}}\,  \|\Phi\|_{L_2^{0,\sigma}}^q.
\]
This completes the proof of \eqref{infty_p}.
\end{proof}

\begin{lemma}      \label{p<q-Lqx(Lpt)}
Let $p,q$ be such that $2\leq p< q<\infty$;  for $\gamma \geq \frac{q-2}{q}$ let $\tilde{\sigma}=\gamma \frac{q}{q-2}\geq 1$. There
 exists a positive constant C such that
 \begin{equation}          \label{Lqx(Lpt)p<q}
E\big( \|D^\gamma_x v\|^q_{L^q_x(L^p_t)} \big)  \leq C\, T^{\frac{q}{p}+1} \|\Phi\|_{L_2^{0,\tilde{\sigma}-1}}^q.
 \end{equation}
\end{lemma}
\begin{proof}
Lemma \ref{lem0} applied with $\sigma=\tilde{\sigma}-1$ yields
\[ \| D^{\tilde{\sigma}}_x v\|_{L^\infty_x(L^q_\omega(L^p_t))}\leq C \, T^{\frac{1}{p}} \, \|\Phi\|_{L_2^{0,\tilde{\sigma}-1}}.
\]
The proof of \eqref{Lqx(Lpt)p<q} relies on the above inequality and on the upper estimate
\begin{equation}    \label{2x-qomega-pt}
\| v\|_{L^2_x(L^q_\omega(L^p_t))} \leq C T^{\frac{1}{p}+\frac{1}{2} } \, \|\Phi\|_{L^{0,0}_2}.
\end{equation}
Indeed, suppose that \eqref{2x-qomega-pt}  has been proved. Since  $\gamma \in [0,\tilde{\sigma}] $,
an interpolation argument (see \cite{deB_Deb_sKdV} Proposition A1)
proves that for $p(\gamma)$ defined by
$\frac{1}{p(\gamma)} = \frac{1}{2} \big( 1- \frac{\gamma}{\tilde{\sigma}}\big)$,  we have $D^\gamma v \in L^{p(\gamma)}_x(L^q_\omega(L^p_t))$.
Note that  $\frac{\gamma}{\tilde{\sigma}} = 1-\frac{2}{q}$;  
hence, $p(\gamma)=q$ and the Fubini theorem
implies that $D^\gamma v \in L^q_\omega(L^q_x(L^p_t))$. Furthermore,
\[ \| D^\gamma v\|_{L^q_\omega(L^q_x(L^p_t))} \leq C \| v\|_{L^2_x(L^q_\omega(L^p_t))}^{1-\frac{\gamma}{\tilde{\sigma}} }\;
 \|D^{\tilde{\sigma}}_x v\|_{L^\infty_x(L^q_\omega(L^p_t))}^{\frac{\gamma}{\tilde{\sigma}}}
 \leq C\,  T^{\frac{1}{p}+\frac{1}{q}}\,  \|\Phi\|_{L_2^{0,\tilde{\sigma}-1}} .
 \]
 Thus,  in order to complete the proof of the lemma, we have to check that \eqref{2x-qomega-pt} holds.
 Since $q\geq p$, H\"older's inequality applied with respect to $dt$ and moments of the stochastic integral imply that for
the  CONS $\{e_j\}_{ j\geq 0}$ of $L^2(\RR)$ in the definition of $W(t)$, we have
 \begin{align*}
\| v\|_{L^2_x(L^q_\omega(L^p_t))}^2   &=
 \int_{\RR} \Big| E\Big( \Big\{ \int_0^T  \Big| \int_0^t S(t-s) \Phi dW(s) \Big|^p \, dt\, \Big\}^{\frac{q}{p}}\Big) \Big|^{\frac{2}{q}} dx \\
 & \leq T^{( \frac{q}{p}-1)\frac{2}{q}} \; \int_{\RR} \Big| E \;  \int_0^T  \big| \int_0^t S(t-s) \Phi dW(s) \big|^q dt \Big|^{\frac{2}{q}} dx\\
 &  \leq  C_q\, T^{ \frac{2}{p}- \frac{2}{q}} \; \int_{\RR} \Big[  \int_0^T  \Big( \sum_{j\geq 0} \int_0^t \big|S(t-s) \Phi e_j \big|^2 ds\Big)^{\frac{q}{2}}
  dt \Big]^{\frac{2}{q}} dx\\
  & \leq  C_q \,  T^{ \frac{2}{p}- \frac{2}{q}} \; \int_{\RR} \Big\|   \sum_{j\geq 0} \int_0^t \big|S(s) \Phi e_j \big|^2 ds \Big\|_{L^{\frac{q}{2}}_t}  dx,
 \end{align*}
 where in the last step we change variable $s$ to $t-s$.
 The Minskowski inequality implies that
 \begin{align*}
\| v\|_{L^2_x(L^q_\omega(L^p_t))}^2
 &\leq C_q \,  T^{ \frac{2}{p}- \frac{2}{q}} \; \int_{\RR} \, \sum_{j\geq 0}  \Big\|    \int_0^t \big|S(s) \Phi e_j \big|^2 ds \Big\|_{L^{\frac{q}{2}}_t}  \, dx\\
 &\leq  C_q \,  T^{ \frac{2}{p}- \frac{2}{q}} \; \int_{\RR} \sum_{j\geq 0}  \Big\{ \int_0^T \Big|    \int_0^T \big|S(s) \Phi e_j \big|^2 ds
 \Big|^{\frac{q}{2}} dt \Big\}^{\frac{2}{q}}   dx\\
&\leq C_q \,  T^{ \frac{2}{p}} \sum_{j\geq 0} \int_0^T \Big( \int_\RR \big|S(s) \Phi e_j \big|^2 dx \Big) ds \leq C\, T^{\frac{2}{p}+1} \sum_{j \geq 0}
\|\Phi e_j\|_{L^2_x}^2.
 \end{align*}
 This completes the proof of \eqref{2x-qomega-pt}.
\end{proof}

The following lemma extends Proposition 3.3 in \cite{deB_Deb_sKdV} to the case $\sigma <\frac{3}{4}$. The notation $a\vee b$
means $\max(a,b)$, while $a\wedge b$ means $\min(a,b)$.
\begin{lemma}       \label{inftyx-2t}
Let $\sigma>0$ and $\epsilon \in (0,  2)\cap (0,\sigma]$. Then there exists a constant $C>0$ such that
\begin{equation}    \label{Linftyx(L2t)}
E\Big( \| D^{\sigma - \epsilon}_x \partial_x v\|_{L^\infty_x(L^2_t)}^2\Big)
 \leq C\, T^2 \, \|\Phi\|^2_{L_2^{0, (\frac{1}{2}-\frac{\epsilon}{4}) \vee \sigma}}.
\end{equation}
Furthermore,
\begin{equation}      \label{LinftyxDx}
E\Big( \| \partial_x  v\|_{L^\infty_x(L^2_t)}^2\Big)
 \leq C\, T^2 \, \|\Phi\|^2_{L_2^{0, \frac{2}{5}}}.
\end{equation}
\end{lemma}
\begin{proof}
We  first prove \eqref{Linftyx(L2t)} and let $q=\frac{4}{\epsilon}$.
H\"older's inequality with respect to the expectation shows that \eqref{Linftyx(L2t)} is  a consequence
of the following estimate
\begin{equation}   \label{qinfty2}
\Big[ E\Big( \Big\{ \sup_{x\in \RR} \int_0^T |D^{\sigma - \epsilon}_x \partial_x v|^2 dt \Big\}^{\frac{q}{2}} \Big) \Big]^{\frac{2}{q}} =
\| D^{\sigma - \epsilon}_x \partial_x v\|_{L^q_\omega(L^\infty_x(L^2_t))}^2 \leq C T^2 \, \|\Phi\|^2_{L_2^{0, (\frac{1}{2}-\frac{\epsilon}{4}) \vee \sigma}}.
\end{equation}
Lemma \ref{lem0} applied with $p=2$ implies
\[ \| D^{1+\sigma}_x v\|_{L^\infty_x(L^q_\omega(L^2_t))} \leq C\, T^{\frac{1}{2}}\, \|\Phi\|_{L_2^{0,\sigma}}.
\]
We next prove that
\begin{equation}  \label{2q2}
\| D^\sigma_x v\|_{L^2_x(L^q_\omega(L^2_t))} \leq C_q T \|\Phi\|_{L^{0,\sigma}_2}.
\end{equation}
Indeed, the two previous estimates imply by interpolation (see \cite{deB_Deb_sKdV} Proposition A1) that, since $q=\frac{4}{\epsilon}$,
we have $\frac{1}{q}=\frac{1}{2}\big( 1- \big( 1-\frac{\epsilon}{2} \big) \big)$, which  yields
\[ \|D^{\sigma +1-\frac{\epsilon}{2}}_x v\|_{L^q_x(L^q_\omega(L^2_t))} \leq C
\|D^{\sigma}_x v\|_{L^2_x(L^q_\omega(L^2_t))}^{\frac{\epsilon}{2}} \;  \| D^{1+\sigma}_x v\|_{L^\infty_x(L^q_\omega(L^2_t))}^{1-\frac{\epsilon}{2}}.
\]
Thus, using the Fubini theorem, we deduce that
\begin{equation}     \label{Dqq2}
\|D^{\sigma +1-\frac{\epsilon}{2}}_x v\|_{L^q_\omega(L^q_x(L^2_t))} \leq C \, T^{\frac{1}{2} + \frac{\epsilon}{4}} \|\Phi\|_{L_2^{0,\sigma}}.
\end{equation}
To prove \eqref{2q2} using H\"older's inequality with respect to $dt$, Fubini's theorem and moments of the stochastic integral, we deduce that
for any CONS $\{e_k\}_{ k\geq 0}$ of $L^2(\RR)$, we have
\begin{align*}
\|D^\sigma_x v\|_{L^2_x(L^q_\omega(L^2_t))}^2 &
=\int_\RR \Big[ E\Big( \Big\{\int_0^T  \Big| D^\sigma_x \int_0^t S(t-s) \Phi dW(s)\Big|^2 dt\Big\}^{\frac{q}{2}} \Big) \Big]^{\frac{2}{q}} \, dx\\
&\leq T^{(\frac{q}{2}-1) \frac{2}{q}} \, \int_{\RR} \Big[ E\Big(   \int_0^T  \Big| D^\sigma_x \int_0^t S(t-s)
\Phi dW(s) \Big|^q dt\Big) \Big]^{\frac{2}{q}}\, dx\\
&\leq T^{1- \frac{2}{q}} \, \int_{\RR}\Big[ \int_0^T E\Big( \Big| D^\sigma_x \int_0^t S(t-s) \Phi dW(s) \Big|^q \Big) dt \Big]^{\frac{2}{q}}\, dx\\
&\leq C_q \, T^{1- \frac{2}{q}} \, \int_{\RR}\Big[ \int_0^T \Big| \sum_{j\geq 0} \int_0^t D^\sigma_x S(t-s) \Phi e_j |^2 ds \Big|^{\frac{q}{2}} dt
\Big]^{\frac{2}{q}}\, dx
\\
&\leq C_q  \, T^{1- \frac{2}{q}} \, T^{\frac{2}{q}} \, \int_{\RR}\Big[ \Big| \sum_{j\geq 0} \int_0^T | D^\sigma_x S(s) \Phi e_j|^2 ds\Big|^{\frac{q}{2} }
\Big]^{\frac{2}{q}}\, dx
 \\
&\leq C_q \, T\, \sum_{j\geq 0} \int_0^T \Big( \int_{\RR} |D^\sigma_x S(s) \Phi e_j|^2 dx \Big) ds
 \; \leq \;  C_q T^2 \sum_{j\geq 0} \|\Phi e_j\|_{H^\sigma_x}^2,
\end{align*}
 which   completes the proof of \eqref{2q2},  and thus, of \eqref{Dqq2}.

We next  compute an upper estimate of  $\| v\|_{L^q_\omega(L^q_x(L^2_t))}$.
Using Fubini's theorem, H\"older's inequality with respect to $dt$ and moments of
the stochastic integral, we obtain
\begin{align*}
\| v\|_{L^q_\omega(L^q_x(L^2_t))}^q &= \int_{\RR} E\Big( \Big\{ \int_0^T \Big| \int_0^t S(t-s) \Phi dW(s)\Big|^2 \, dt \Big\}^{\frac{q}{2}} \Big) \, dx\\
&\leq T^{\frac{q}{2}-1}\, \int_{\RR} E\Big(  \int_0^T \Big| \int_0^t S(t-s) \Phi dW(s)\Big|^q \, dt  \Big) \, dx
 \\
&\leq C_q \, T^{\frac{q}{2}-1}\, \int_{\RR} \Big[ \int_0^T \Big( \sum_{j\geq 0} \int_0^t |S(t-s)\Phi e_j |^2 ds \Big)^{\frac{q}{2}} dt \Big] dx
\\
& \leq  C_q \, T^{\frac{q}{2}}\, \int_{\RR} \Big( \sum_{j\geq 0} \int_0^T |S(s)\Phi e_j |^2 ds \Big)^{\frac{q}{2}}\, dx.
\end{align*}
The Sobolev embedding theorem implies that for $\tilde{\sigma}=\frac{1}{2}-\frac{1}{q}$, we have $H_x^{\tilde{\sigma}}\subset L^q_x$.
Therefore, Minkowski's inequality yields
\begin{align}    \label{qq2}
\| v\|_{L^q_\omega(L^q_x(L^2_t))}^2
&\leq  C_q \, T \Big\{ \int_{\RR} \Big( \sum_{j\geq 0} \int_0^T |S(s) \Phi e_j |^2 ds \Big)^{\frac{q}{2}} \, dx \Big\}^{\frac{2}{q}}
\nonumber \\
&\leq C_q \, T\, \sum_{j\geq 0}  \int_0^T \Big\|  |S(s) \Phi e_j|^2  \Big\|_{L^{\frac{q}{2}}_x} \, ds \nonumber  \\
&\leq C_q \, T\, \sum_{j\geq 0}  \int_0^T \big\|  S(s) \Phi e_j  \big\|_{L^{q}_x}^2 \, ds \nonumber  \\
&\leq C_q \, T^2 \, \sum_{j\geq 0} \sup_{s\in [0,T]} \big\|S(s) \Phi e_j  \big\|^2_{H^{\tilde{\sigma}}_x }  =
 C\, T^2\, \|\Phi\|_{L^{0, \frac{1}{2}-\frac{\epsilon}{4}}_2}^2.
\end{align}
The inequalities \eqref{qq2} and \eqref{Dqq2} imply that
\[ \| v \|_{L^q_\omega(W_x^{\sigma +1-\frac{\epsilon}{2},q}(L^2_t))} \leq C\, T\, \|\Phi\|_{L_2^{0, (\frac{1}{2}-\frac{\epsilon}{4})\vee \sigma}}.\]
Since $q\frac{\epsilon}{2}=2 \;   \geq 1 $, the Sobolev embedding theorem yields
$W_x^{\frac{\epsilon}{2},q}(L^2_t)\subset L^\infty_x(L^2_t)$; thus,
$D^{1+\sigma-\epsilon}_x v\in L^q_\omega(L^\infty_x(L^2_t))$ and
\begin{equation}   \label{D1+sigma-epsilon}
 \|D^{1+\sigma-\epsilon}_x v\|_{L^q_\omega(L^\infty_x(L^2_t))} \leq C\, T\, \|\Phi\|_{L_2^{0, (\frac{1}{2}-\frac{\epsilon}{4})\vee \sigma}}.
 \end{equation}
Finally,
\[ D^{\sigma -\epsilon}_x \partial_x v= \int_0^t D^{\sigma-\epsilon}_x \partial_x S(t-s)\Phi dW(s) =  \int_0^t D^{1+ \sigma-\epsilon}_xS(t-s)
 {\mathcal H} \Phi dW(s),\]
where ${\mathcal H}$ denotes the Hilbert transform. Thus, we obtain
\[ \| D^{\sigma - \epsilon}_x \partial_x v\|_{ L^q_\omega(L^\infty_x(L^2_t))} \leq C\, T\, \|{\mathcal H} \Phi\|_{L_2^{0, (\frac{1}{2}-\frac{\epsilon}{4})
\vee \sigma}} \leq   C\, T\,  \|\Phi\|_{L_2^{0, (\frac{1}{2}-\frac{\epsilon}{4})
\vee \sigma}}.\]
This completes the proof of \eqref{qinfty2},  and therefore, of \eqref{Linftyx(L2t)}.

To prove \eqref{LinftyxDx}, let $\sigma=\epsilon=\frac{2}{5}$. Then $\frac{1}{2}-\frac{\epsilon}{4} =\sigma$ and
 \eqref{D1+sigma-epsilon} completes the proof.
\end{proof}

\begin{lemma}    \label{q<p-Lqx(Lpt)}
Let $p,q$ be such that $2\leq q\leq p<\infty$ and $\gamma \geq 0$. There exists a constant $C>0$ such that
\begin{equation}       \label{Lqx(Lpt)q<p}
E\|D_x^\gamma v\|^q_{L^q_x(L^p_t)} \leq C\, T^{\frac{q}{p}+\frac{q}{2}} \, \|\Phi\|^q_{L_2^{0, \gamma + \frac{1}{2} - \frac{1}{q}}}.
\end{equation}
\end{lemma}
\begin{proof}
Fubini's theorem and H\"older's inequality with respect to $dt$   prove that
\[  E\|D_x^\gamma v\|^q_{L^q_x(L^p_t)} = \| D_x^\gamma v\|^q_{L^q_x(L^q_\omega(L^p_t))} \leq \|D^\gamma v\|_{L^q_x(L^p_\omega (L^p_t))}
= \|D^\gamma v\|^q_{L^q_x(L^p_t (L^p_\omega))}.
\]
Hence, \eqref{Lqx(Lpt)q<p} can be obtained from the following estimate
\begin{equation}    \label{qpp}
\|D^\gamma_x v\|^q_{L^q_x(L^p_t (L^p_\omega))}
 \leq C\, T^{\frac{q}{p}+\frac{q}{2}} \, \|\Phi\|^q_{L_2^{0,\gamma + \frac{1}{2} - \frac{1}{q}}}.
\end{equation}
Moments of the stochastic integral, a change of variables  and H\"older's inequality with respect  to $ds$
 imply that for the CONS $\{e_j\}_{j\geq 0}$ of $L^2(\RR)$ in the definition of $W$, we have
\begin{align*}
\|D^\gamma v\|^q_{L^q_x(L^p_t (L^p_\omega))}
& = \int_\RR \Big| \int_0^T E\Big( \Big| D^\gamma_x \int_0^t  S(t-s) \Phi dW(s)\Big|^p \Big) dt \Big|^{\frac{q}{p}} dx \\
&=C_p \int_\RR \Big( \int_0^T  \Big| \sum_{j\geq 0} \int_0^t |D^\gamma_x S(t-s) \Phi e_j|^2 ds \Big|^{\frac{p}{2}} dt \Big)^{\frac{q}{p}} dx \\
&\leq C_p \int_\RR \Big( \int_0^T  \Big| \sum_{j\geq 0} \int_0^T |D^\gamma_x S(s) \Phi e_j|^2 ds \Big|^{\frac{p}{2}} dt \Big)^{\frac{q}{p}} dx \\
& \leq C_p \int_\RR \, T^{\frac{q}{p}} \, \Big( \int_0^T \sum_{j\geq 0} |D^\gamma_x S(s) \Phi e_j|^2 ds \Big)^{\frac{q}{2}} \, dx \\
&\leq C_p \,  T^{\frac{q}{p}}\, T^{(1-\frac{2}{q})\frac{q}{2}} \, \int_\RR \Big( \int_0^T
\Big|\sum_{j\geq 0} |D^\gamma_x S(s) \Phi e_j|^2 \Big|^{\frac{q}{2}} ds\Big) dx .
\end{align*}
Using the Fubini theorem and  then the Minkowski inequality, we deduce
   \begin{align*}
 \int_\RR \Big( &\int_0^T  \Big|\sum_{j\geq 0} |D^\gamma_x S(s) \Phi e_j|^2 \Big|^{\frac{q}{2}} ds\Big) dx
=   \int_0^T \Big\| \sum_{j\geq 0} |D^\gamma_x S(s)\Phi e_j|^2 \|_{L^{\frac{q}{2}}_x}^{\frac{q}{2}} \, ds\\
&\leq   \int_0^T \Big( \sum_{j\geq 0}
\| \, D^{\gamma}_x S(s) \Phi e_j|^2 \, \|_{L^{\frac{q}{2}}_x} \Big)^{\frac{q}{2}} ds 
 =  \int_0^T \Big( \sum_{j\geq 0}  \|  D^{\gamma}_x S(s) \Phi e_j  \|^2_{L^q_x} \Big)^{\frac{q}{2}} ds \\
&\leq C  \int_0^T \Big( \sum_{j\geq 0}  \|  D^{\gamma}_x S(s) \Phi e_j  \|^2_{H^\sigma_x} \Big)^{\frac{q}{2}} ds 
= C \, T\,  \|\Phi\|_{L_2^{0, \gamma + \frac{1}{2} -\frac{1}{q}}}^q,
\end{align*}
where in the last line we use  the Sobolev embedding
$H^\sigma_x \subset L^q_x$ for $\sigma = \frac{1}{2} - \frac{1}{q}$.
This completes the proof.
\end{proof}

Finally, in the case of the stochastic mKdV equation, we have to prove a result similar to Proposition 3.2 in \cite{deB_Deb_sKdV}. However,  we have to estimate the  $L^4_x(L^\infty_t)$ norm instead of the $L^2_x(L^\infty_t)$; this
requires a stronger condition on the operator $\Phi$ which has to be in $L^{0,1+\epsilon}_2$ for some positive $\epsilon$.
\begin{lemma}     \label{lem-4xinftyt}
Let $\Phi \in L^{0,1+\epsilon}_2$ for some positive $\epsilon$. Then $v\in L^4_\omega(L^4_x(L^\infty_t))$ and there exists a positive constant
$C$ such that
\begin{equation}     \label{4xinftyt}
E\Big( \int_{\RR}\;  \sup_{t\in [0,T]} \, \Big| \int_0^t S(t-s) \Phi dW(s) \Big|^4 \; dx\Big) \leq C \,   (T+T^{4})   \, \|\Phi\|_{L^{0,1+\epsilon}_2}^4.
\end{equation}
\end{lemma}
\begin{proof} The proof is  based on results from the proof  of Proposition 3.2 in \cite{deB_Deb_sKdV}. We send to this
reference for   some intermediate results.   Let $\{e_j\}_{ j\geq 0}$ be the CONS of $L^2(\RR)$ in the definition of $W$.
Let $\{\psi_k\}_{ k\geq 0}$ denote a partition of unity
such that
\[ \mbox{\rm supp } \psi_0\subset [-1,+1],\quad
\mbox{\rm supp } \psi_k\subset [2^{k-2},2^k] \; 
,\quad  \psi_k(\xi)=\psi_1\Big( \frac{\xi}{2^{k-1}}\Big) \; \mbox{\rm for }\xi\geq 0, \,  k\geq 1.
\]
Let $\tilde{\psi}_k \in C^\infty_0(\RR)$ satisfy $\tilde{\psi}_k\geq 0$, $\tilde{\psi}_k=1$ on the support of $ \psi_k$,  and
$\mbox{\rm supp } \tilde{\psi}_k \subset [2^{k-3}, 2^{k+1}]$. For $k\in \N$ let $S_k(t)$ and $\Phi_k$ be defined by
\begin{align*}  
\widehat{S_k(t) u}(\xi) &= \psi_k(|\xi|) \; \widehat{S(t) u}(\xi) = e^{it\xi^3} \; \psi_k(|\xi|)\;  \hat{u}(\xi),\\
\widehat{\Phi_k e_j}(\xi)&= \tilde{\psi}_k(|\xi|) \; \widehat{\Phi e_j}(\xi), \quad j\in \N.
\end{align*}
Then $S_k(t)\Phi = S_k(t) \Phi_k$, $k\in \N$  and $S(t) \Phi = \sum_{k\geq 0} S_k(t) \Phi_k$.
We prove that for every $k\in \N$ and  $ \epsilon \in (0,1)$, 
\begin{equation}    \label{estim_infty_t_k}
E\Big( \int_{\RR} \, \sup_{t\in [0,T]} \; \Big| \int_0^t S_k(t-s) \Phi_k dW(s)\Big|^4\; dx \Big) \leq C\, (T+T^{4})\, 2^{\epsilon k}\,
 \Big( \sum_{j\in \N} \|\Phi_k e_j\|_{H^{1+\frac{\epsilon}{2}}_x}^2\Big)^2.
\end{equation}
Suppose that \eqref{estim_infty_t_k} holds. Then using the Minkowski and Cauchy-Schwarz inequalities,  we deduce that
\begin{align*}
\Big\{ E\Big( \int_{\RR}& \, \sup_{t\in [0,T]} \; \Big| \int_0^t S(t-s) \Phi dW(s)\Big|^4\; dx \Big) \Big\}^{\frac{1}{4}} \\
& = \Big\{ E\Big( \int_{\RR} \, \sup_{t\in [0,T]} \;  \Big| \sum_{k\in \N}  \int_0^t S_k(t-s) \Phi_k dW(s)\Big|^4\; dx \Big)\Big\}^{\frac{1}{4}} \\
&\leq \sum_{k\in \N} \, \Big\{ E\Big( \int_{\RR} \,   \sup_{t\in [0,T]} \; \Big| \int_0^t S_k(t-s) \Phi_k dW(s)\Big|^4\; dx \Big) \Big\}^{\frac{1}{4}}  \\
&\leq C \, (T+T^4)^{\frac{1}{4}} \, \sum_{k\in \N} 2^{\frac{\epsilon k}{4} }\; \Big( \sum_{j\in \N}
\|\Phi_k e_j \|_{H^{1+\frac{\epsilon}{2}}_x}^2\Big)^{\frac{1}{2}}\\
&\leq C\,  (T+T^4)^{\frac{1}{4}}  \, \Big( \sum_{k\in \N} 2^{-\frac{\epsilon k}{2}} \Big)^{\frac{1}{2}}
\Big( \sum_{k\in \N} 2^{\epsilon k}
\Big\{ \sum_{j\in \N}  \|\Phi_k e_j \|_{H^{1+\frac{\epsilon}{2}}_x}^2\ \Big\} \Big)^{\frac{1}{2}} \\
& \leq C \; (T+T^4)^{\frac{1}{2}}
\| \Phi_k\|_{L^{0,1+\epsilon}_2};
\end{align*}
the last inequality is obtained from  the upper estimate 
$ \sum_{k\in \N}  2^{\epsilon k} \|\Phi_k \varphi\|_{H^{1+\frac{\epsilon}{2}}_x}^2 \leq C\,  \|\Phi_k \varphi\|_{H^{1+\epsilon}_x}^2$
for every $\varphi\in L^2_x$.

We next prove \eqref{estim_infty_t_k}. Let  $\alpha>0$ to be chosen later,  and $p\geq 4$ such that
 $\alpha p>1$. The Sobolev embedding implies that
$W_t^{\alpha,p}\subset L^\infty_t$; hence, using Fubini's theorem we obtain
\begin{equation}    \label{I1+I2}
E\Big( \int_{\RR} \, \sup_{t\in [0,T]} \; \Big| \int_0^t S_k(t-s) \Phi_k dW(s)\Big|^4\; dx \Big) \leq C (I_1+I_2),
\end{equation}
where
\begin{align*}
I_1&= \int_{\RR} E\Big(  \Big\{ \int_0^T \int_0^T \frac{  \big| \int_0^t S_k(t-s) \Phi_k dW(s) -  \int_0^{t'} S_k(t'-s) \Phi_k dW(s) \big|^p }
{|t-t'|^{1+\alpha p}} \, dt\, dt' \Big\}^{\frac{4}{p}}\Big)  dx, \\
I_2&=\int_{\RR} E\Big( \Big\{ \int_0^T \Big|  \int_0^t S_k(t-s) \Phi_k dW(s)\Big|^p dt \Big\}^{\frac{4}{p}} \Big) dx.
\end{align*}
To upper estimate $I_2$, we use H\"older's inequality with respect to the expected value, Fubini's theorem, moments of Gaussian variables
and   Minkowski's    inequality with respect to $dt$ and $dx$;
this yields
\begin{align*}
I_2&\leq \int_{\RR} \Big\{ E\Big( \int_0^T \Big|  \int_0^t S_k(t-s) \Phi_k dW(s)\Big|^p dt\Big) \Big\}^{\frac{4}{p}} dx\\
& \leq C_p  \int_{\RR} \Big\{ \int_0^T \Big| \sum_{j\in \N} \int_0^t \big| S_k(t-s) \Phi_k e_j\big|^2 ds \Big|^{\frac{p}{2}} dt \Big\}^{\frac{4}{p}} dx\\
&\leq C_p \int_{\RR} \Big\{  \sum_{j\in \N} \Big[ \int_0^T \Big| \int_0^T \big| S_k(s) \Phi_k e_j\big|^2 ds \Big|^{\frac{p}{2}} dt \Big]^{\frac{2}{p}}
\Big\}^2 dx\\
&\leq C_p T^{\frac{4}{p}} \int_{\RR} \Big\{  \sum_{j\in \N}  \Big| \int_0^T \big| S_k(s) \Phi_k e_j\big|^2 ds \Big\}^2 dx\\
&\leq C_p T^{2+\frac{4}{p}}  \Big\{  \sum_{j\in \N}  \sup_{s\in [0,T]} \| | S_k(s) \Phi_k e_j|^2 \|_{L^2_x} \Big\}^2\\
&\leq  C_p T^{2+\frac{4}{p}}  \Big\{  \sum_{j\in \N}  \sup_{s\in [0,T]} \|  S_k(s) \Phi_k e_j \|_{L^4_x}^2 \Big\}^2\\
&\leq  C T^{2+\frac{4}{p}}  \Big\{  \sum_{j\in \N}  \sup_{s\in [0,T]} \|  S_k(s) \Phi_k e_j \|_{H^\sigma_x}^2 \Big\}^2,
\end{align*}
where the last inequality can be deduced from the inclusion $H^\sigma_x\subset L^4_x$ for $\sigma > \frac{1}{4}$ to be chosen later.

\begin{remark}
This is the place where the significant difference with the stochastic KdV case in \cite{deB_Deb_sKdV} arises. Indeed,
to deal with the higher  power of   nonlinearity,  the functional space here is $L_x^4(L^\infty_t)$  instead of $L_x^2(L^\infty_t)$.
\end{remark} 

 Using Theorem 2.7 in \cite{KPV_JAMS}, we  first consider the homogeneous part of the $H^\sigma_x$-norm
(denoted by $\dot{H}^\sigma_x$). For   $\tau > \frac{3}{4}$, if $\sigma = \tau - \frac{1}{2}$,  we obtain
\begin{equation}  \label{upperI2}
  \Big\{  \sum_{j\in \N}  \sup_{s\in [0,T]} \|  S_k(s) \Phi_k e_j \|_{\dot{H}^\sigma_x}^2 \Big\}^2
   \leq C   \Big\{  \sum_{j\in \N}  \| D_x^\sigma \Phi_k e_j \|_{H^\tau_x}^2 \Big\}^2 \leq C
\|\Phi_k\|_{L^{0,\sigma+\tau}_2}^4.
\end{equation}
The $L^2_x$ part of the $H^\sigma_x$-norm obviously satisfies the same final upper bound.

To upper estimate $I_1$, we use H\"older's inequality with respect to the expected value and Fubini's theorem,
\[ 
I_1\leq \int_{\RR} \Big\{ \int_0^T \int_0^T \frac{ E\big( \big| \int_0^t S_k(t-s) \Phi_k dW(s) -  \int_0^{t'} S_k(t'-s) \Phi_k dW(s) \big|^p\big)  }
{|t-t'|^{1+\alpha p}} \, dt\, dt' \Big\}^{\frac{4}{p}} dx\; .
\]
Since the stochastic integral is Gaussian, for $t \leq t'$ we have
\begin{align*}
 E\Big( &\Big| \int_0^t S_k(t-s) \Phi_k dW(s) -  \int_0^{t'} S_k(t'-s) \Phi_k dW(s)\Big|^p \Big) \\
 & = C_p
\Big|\sum_{j\in \N}  \int_0^t \big| S_k(t-s) \Phi_k  e_j -  S_k(t'-s) \Phi_k e_j|^2 ds \Big|^{\frac{p}{2}}
+ C_p  \Big| \int_t^{t'} \sum_{j\in \N} \big| S_k(t'-s) \Phi_k e_j  |^2 ds \Big|^{\frac{p}{2}} .
\end{align*}
 In the double time integral  we first   consider the case $|t-t'| 2^{\gamma k} \leq 1$ for $\gamma>0$ to be chosen later on.
Using parts of the proof of Proposition 3.1 pages 228-229  in \cite{deB_Deb_sKdV} based on computations from \cite{KPV_JAMS}, we deduce that
for $k,j\in \N$  and  $0\leq t\leq t'\leq T$, we obtain
\begin{align*}
 & \int_0^t \big| S_k(t-s) \Phi_k  e_j -  S_k(t'-s) \Phi_k e_j|^2 ds \leq C \big( |t-t'| 2^{3k} + |t-t'|^2 2^{5k}\big)^2 \big( H_k^T \ast |\Phi_k e_j|\big)^2,\\
 &\int_t^{t'} \big| S_k(t'-s) \Phi_k e_j|^2 ds \leq C |t'-t| \big( H_k^T \ast |\Phi_k e_j|\big)^2,
 \end{align*}
where for $k\geq 1$  (resp. $k=0$)  we let
\begin{align*}
H_k^T(x)&= 2^{k-1}\; \mbox{\rm for } |x|\leq C_1 (T+1), \;
H_k^T(x)=\frac{  2^{\frac{k-1}{2}}}{  |x|^{\frac{1}{2}}} \; \mbox{\rm for }  C_1 (T+1) < |x|\leq C_2(T+1) 2^{2(k-1)},\\
H_k^T(x)&= \frac{1}{1+x^2} \; \mbox{\rm for }  |x| >  C_2(T+1) 2^{2(k-1)},\\
H_0^T(x)&= 1 \; \mbox{\rm for } |x|\leq C_1 (T+1), \quad H_0^T(x)= \frac{1}{1+x^2} \; \mbox{\rm for } |x| > C_1 (T+1).
\end{align*}
Hence,  we deduce that for $k \in \N$ and $0\leq t\leq t'\leq T$, we get
\begin{align*}
 J_{k}(t,t')& : = E\Big( \Big| \int_0^t  S_k(t-s) \Phi_k dW(s) -  \int_0^{t'} S_k(t'-s) \Phi_k dW(s)\Big|^2 \Big) \\
 & \leq C \big( |t-t'| +
|t-t'|^2 2^{6k} + |t-t'|^4 2^{10k}\big) \; \sum_{j\in \N} \big( H_k^T \ast |\Phi_k e_j|\big)^2.
\end{align*}
Fix $\epsilon \in (0,1)$;  choose  $\gamma>\frac{9}{2} $  and $\alpha <\frac{1}{8}$ such that $\alpha \gamma < \frac{\epsilon}{8}$.
Note that for $\epsilon \in (0,1)$ we have $\alpha <\frac{1}{36}$; thus, $p>36$. Then for $|t-t'| 2^{\gamma k} \leq 1$, we obtain
\begin{align*}
 J_{k}(t,t') & \leq C \, |t-t'|^{4\alpha} 2^{(-3+4\alpha \gamma) k} \Big( 2^{k(-\gamma +3) } + 2^{k(-2\gamma +9)} + 2^{k(-4\gamma +13)} \Big) \;
\sum_{j\in \N} \big( H_k^T \ast |\Phi_k e_j|\big)^2\\
&\leq C \, |t-t'|^{4\alpha}\,  2^{(-3+ \frac{\epsilon}{2} ) k} \;  \sum_{j\in \N} \big( H_k^T \ast |\Phi_k e_j|\big)^2.
\end{align*}
A direct computation shows that for $k,j\in \N$ and $0\leq t\leq t'\leq T$ such that $|t-t'|\, 2^{\gamma k}>1$, we get
\begin{align*}
 J_k (t,t')&\leq 2\sum_{j\in\N} \Big[ \int_0^t \big| S_k(t-s) \Phi_k e_j|^2 ds +  \int_0^{t'} \big| S_k(t'-s) \Phi_k e_j|^2 ds \Big]\\
 &\leq 4 |t-t'|^{4\alpha} \, 2^{\frac{\epsilon}{2}  k} \sum_{j\in\N} \int_0^T \big| S_k(s) \Phi_k e_j|^2 ds.
\end{align*}
The above upper estimates and Minkowski's inequality with respect to $dx$ imply
\begin{align*}
I_1 \leq &\;  C\, 2^{\epsilon k}\, 2^{-6k} \int_{\RR}
 \Big\{ \int_0^T \int_0^T 1_{\{ |t-t'| 2^{\gamma k}\leq 1\} } \frac{1}{|t-t'|^{1-\alpha p}  } \, dt\,   dt' \Big\}^{\frac{4}{p}} \;
  \Big(  \sum_{j\in \N} \big( H_k^T \ast |\Phi_k e_j|\big)^2\Big)^2 \, dx \\
&\; + C \, 2^{\epsilon k}\, \int_{\RR} \Big\{ \int_0^T \int_0^T 1_{\{ |t-t'| 2^{\gamma k}>  1\} } \frac{1}{|t-t'|^{1-\alpha p}  } \, dt\,   dt' \Big\}^{\frac{4}{p}} \;
\Big( \sum_{j\in \N}  \int_0^T \big| S_k(s) \Phi_k e_j|^2 ds \Big)^2\, dx\\
\leq &\;     C\,  2^{\epsilon k} \,
\Big[2^{-6k}\, T  \Big( \sum_{j\in \N} \big\|    H_k^T \ast |\Phi_k e_j| \big\|_{L^4_x}^2\Big)^2  +
T^2\,  \Big( \sum_{j\in \N} \sup_{s\in [0,T]}  \big\| S_k(s) \Phi_k e_j  \big\|_{L^4_x}^2 \Big)^2 \Big].
\end{align*}
Young's inequality yields
\[ \big\|    H_k^T \ast |\Phi_k e_j| \big\|_{L^4_x} \leq C \|H_k^T\|_{L^{\frac{4}{3}}_x} \, \|\Phi_k e_j\|_{L^2_x}.\]
Furthermore, using the explicit definition of $H^T_k$, we deduce
\[  \big\|    H_k^T \big\|_{L^{\frac{4}{3}}_x}^{\frac{4}{3}} \leq C\, (1+T) \, 2^{\frac{4}{3} k},
 \]
 which implies
\[ \Big( \sum_{j\in \N} \big\|    H_k^T \ast |\Phi_k e_j| \big\|_{L^4_x}^2\Big)^2 \leq C\, \big(1+T^3 \big)\, 2^{4k} \|\Phi_k\|_{L^{0,0}_2}^4.
\]
The upper estimate \eqref{upperI2} implies that  for  $\tau >\frac{3}{4}$ and $\sigma = \tau - \frac{1}{2}>\frac{1}{4}$, we have
\begin{align}   \label{upperI1}
I_1 &\leq C\, 2^{\epsilon k}\, \Big[  (T+T^4)\,
2^{-2k} \|\Phi_k\|_{L^{0,0}_2}^4 + T^2  \|\Phi_k\|_{L^{0,\sigma+\tau}_2}^4\Big]
\leq C \,   (T+T^4)\, 2^{\epsilon k}\, \|\Phi_k\|_{L^{0,\sigma+\tau}_2}^4  .
\end{align}
Since $p>36$, choosing $\sigma$ and $\tau$ such that $\sigma +\tau \leq 1+\frac{\epsilon}{2}$,
the inequalities \eqref{I1+I2}-\eqref{upperI1} conclude the proof of \eqref{estim_infty_t_k},   and thus of the lemma.
\end{proof}

In order to prove the existence of a local solution to \eqref{mild_gKdV}, we  first estimate moments of functional norms
$\|v\|_{{\mathcal X}_k^T}$ of the stochastic integral $v(t)=\int_0^t S(t-s) \Phi dW(s)$, $k=2,3$.
Let $ u\in {\mathcal X}_k^T$; following the notations in \cite{KPV_JAMS}, we   set
\begin{equation}            \label{normsX2-3}
\| u\|_{{\mathcal X}_2^T} = \max_{j=1, \cdots, 5} \mu_j^T(u) \quad
\mbox{\rm ( resp. } \| u\|_{{\mathcal X}_3^T} = \max_{j=1, \cdots, 7} \nu_j^T(u) \mbox{\rm )},
 \end{equation}
where for some positive number $\rho$, we define
\begin{align*}
\mu_1^T(u)&=\sup_{t\in [0,T]} \| D^{\frac{1}{4}}_x u(t)\|_{L^2_x}, \quad \mu_2^T(u)=\|D_x u\|_{L^{20}_x(L^\frac{5}{2}_t)}, \quad
\mu_3^T(u)=\|D^{\frac{1}{4}}_x u\|_{L^5_x(L^{10}_t)},\\
\mu_4^T(u)&=\| D^{\frac{1}{4}}_x \partial_x u\|_{L^\infty_x(L^2_t)}, \quad \mu_5^T(u)=\|u\|_{L^4_x(L^\infty_t)},\\
\nu_1^T(u)&= \sup_{t\in [0,T]} \|  u(t)\|_{H^{\frac{1}{12}}_x}, \quad \nu_2^T(y)=(1+T)^{-\rho} \|u\|_{L^{\frac{42}{13}}_x(L^{\frac{21}{4}}_t)},
\quad \nu_3^T(u)=\|u\|_{L^{\frac{60}{13}}_x(L^{15}_t)},\\
\nu_4^T(u)&=T^{-\frac{1}{6}} \|u\|_{L^{\frac{10}{3}}_x(L^{\frac{30}{7}}_t)},\quad \nu_5^T(u)=\nu_4^T(D^{\frac{1}{12}}_x u),\\
\nu_6^T(u)&=\|\partial_x u\|_{L^\infty_x(L^2_t)}, \quad \nu_7^T(u)=\nu_6^T(D_x^{\frac{1}{12}} u).
\end{align*}
The following proposition gathers the information from the previous lemmas.
\begin{Prop}    \label{SI_Xk}
For $t\in [0,T]$ let $v(t)=\int_0^t S(t-s)\, \Phi\, dW(s)$.
\begin{enumerate}
  \item[(i)] 
Suppose that $\Phi\in L^{0,1+\epsilon}_2$ for some $\epsilon >0$.  Then
for some positive constant $C$, we have
\begin{equation}    \label{SI_X2}
E\big( \|v\|_{{\mathcal X}_2^T}^2 \big)  \leq C\, (\sqrt{T}+T^2)\,  \|\Phi\|_{L^{0,1+\epsilon}_2}^2.
\end{equation}

  \item[(ii)]
Suppose that $\Phi\in L^{0,\frac{5}{12}}_2$. Then
for some positive constant $C$, we obtain
\begin{equation}    \label{SI_X3}
E\big( \|v\|_{{\mathcal X}_3^T}^2  \big) \leq C \, (T+T^2)\,  \|\Phi\|_{L^{0,\frac{5}{12}}_2}^2 .
\end{equation}
\end{enumerate}
\end{Prop}

\begin{proof}
(i) Consider  $k=2$  (mKdV).

Lemma \ref{lemunifH} applied  with $q=1$ and $\sigma=\frac{1}{4}$  implies that
$E\big( \big|\mu_1^T(v)\big|^2 \big)\leq C\, T\, \|\Phi\|_{L^{0,\frac{1}{4}}_2}^2$. 

Using Lemma \ref{p<q-Lqx(Lpt)} with
$p=\frac{5}{2}<20=q$ and $\gamma =1$,
we obtain $  E\big( \big|\mu_2^T(v)\big|^{20} \big)\leq C\, T^9\, \|\Phi\|_{L^{0,\frac{1}{9}}_2}^{20}$.

Lemma \ref{q<p-Lqx(Lpt)} applied with $\gamma = \frac{1}{4}$, $2< q=5<p=10$ yields
$ E\big( \big|\mu_3^T(v)\big|^5 \big)\leq C\, T^3 \, \|\Phi\|_{L^{0,\frac{11}{20}}_2}^5$.

Lemma \ref{inftyx-2t}   applied with $\sigma = \frac{9}{20}$ and $\epsilon = \frac{1}{5}$ yields
$ E\big( \big|\mu_4^T(v)\big|^2 \big)\leq C\, T^2  \, \|\Phi\|_{L^{0,\frac{9}{20}}_2}^2 $.   

Finally, Lemma \ref{lem-4xinftyt} implies
$ E\big( \big|\mu_5^T(v)\big|^4 \big)\leq C(T+T^4)  \, \|\Phi\|_{L^{0,1+\epsilon }_2}^2$ for any $\epsilon >0$.

These estimates and H\"older's inequality  conclude the proof of  \eqref{SI_X2}.
\smallskip

(ii) Consider  $k=3$ (gKdV).  

Lemma \ref{lemunifH} applied with $q=1$ and $\sigma=\frac{1}{12}$ implies that
$E\big( \big|\nu_1^T(v)\big|^2 \big)\leq C\, T\, \|\Phi\|_{L^{0,\frac{1}{12}}_2}^2$ .

Apply Lemma \ref{q<p-Lqx(Lpt)} to upper estimate moments of $\nu_k^T(v)$ for $k=2, ... , 5$.
Take  $\gamma=0$, and either $2\leq q=\frac{42}{13}<p=\frac{21}{4}$ for $\nu_2^T(v)$ or
$2\leq q=\frac{60}{13}<p=15$ for $\nu_3^T(v)$. This yields
\begin{align*}
 E\big( \big|\nu_2^T(v)\big|^{\frac{42}{13}} \big)&\leq C\, (1+T)^{-\frac{42 \rho}{13}} T^{\frac{42}{13}\big( \frac{4}{21}+\frac{1}{2}\big)} \,
\|\Phi\|_{L^{0,\frac{4}{21}}_2}^{\frac{42}{13}} \leq   C\,T^{\frac{29}{13}} \,  \|\Phi\|_{L^{0,\frac{4}{21}}_2}^{\frac{42}{13}},   \\
E\big( \big|\nu_3^T(v)\big|^{\frac{60}{13}} \big)&\leq C\, T^{\frac{34}{13}}\,  \|\Phi\|_{L^{0,\frac{17}{60}}_2}^{\frac{60}{13}}.
\end{align*}
Take $2\leq q=\frac{10}{3}<p=\frac{30}{7}$, and either $\gamma=0$ for $\nu_4^T(v)$ or $\gamma = \frac{1}{12}$ for $\nu_5^T(v)$.
This yields
\[
E\big( \big|\nu_4^T(v)\big|^{\frac{10}{3}} \big) \leq
 C\, T^{\frac{41}{18}}\,  \|\Phi\|_{L^{0,\frac{1}{5}}_2}^{\frac{10}{3}},
\quad E\big( \big|\nu_5^T(v)\big|^{\frac{10}{3}} \big)\leq C\, T^{\frac{41}{18}}\,  \|\Phi\|_{L^{0,\frac{17}{60}}_2}^{\frac{10}{3}}  .
\]
Furthermore, the inequality \eqref{LinftyxDx} from Lemma \ref{inftyx-2t} gives exactly
$ E\big( \big|\nu_6^T(v)\big|^2 \big) \leq C\, T^2 \, \|\Phi\|_{L^{0,\frac{2}{5}}_2}^2$.
Finally, the inequality \eqref{Linftyx(L2t)} from  Lemma \ref{inftyx-2t} applied with    $\sigma=\frac{5}{12}$ and $\epsilon = \frac{1}{3}$ yields
$$
E\big( \big|\nu_7^T(v)\big|^2 \big) \leq C\, T^2 \, \|\Phi\|_{L^{0,\frac{5}{12}}_2}^2.
$$
These bounds and H\"older's inequality complete the proof of  \eqref{SI_X3}.
\end{proof}

\section{Local well-posedness} \label{Local}
In this section, we prove the existence of a unique local solution $u \in {\mathcal X}_k^{T(\omega)}$  to \eqref{stoch_gKdV}
for some  {\it random terminal} time $T(\omega)$, which is positive for almost every $\omega$.
\begin{Prop}   \label{local_well_posed}
Let $k=2$,  $u_0\in H^{\frac{1}{4}}_x$ a.s.
 and $\Phi\in L^{0,1+\epsilon}_2$ for some positive $\epsilon$ (resp.  $k=3$,  $u_0\in H^{\frac{1}{12}}_x$ a.s.
 and $\Phi\in L^{0,\frac{5}{12}}_2$). Almost surely there exists a positive random time $T^k(\omega)$, $k=2,3$ such that there exists a unique
 solution to \eqref{stoch_gKdV} in ${\mathcal X}_k^{T_k(\omega)}$.
\end{Prop}
\begin{proof}
Set
$   \sigma(2)=\frac{1}{4} $ 
and $ \sigma(3)=\frac{1}{12}$.   
  Suppose that a.s. $u_0\in H^{\sigma(k)}_x$ for $k=2,3$.
  Using the inequalities (3.6)-(3.7), (3.9), (3.11) and (3.35)  (resp. (3.6)-(3.7), (3.48), (3.52)-(3.53)) in \cite{KPV_CPAM}, we obtain that
for almost every $\omega$,  $S(t) u_0(\omega) \in {\mathcal X}_k^T$  for $u_0(\omega) \in H^{\sigma(k)}_x$. 
Furthermore, $S(.)\big(u_0(\omega)\big) \in C([0,T];H^{\sigma(k)}_x)$
and
 \[ \|S(.) u_0(\omega)\|_{{\mathcal X}_k^T} \leq c_k \|u_0(\omega)\|_{H^{\sigma(k)}_x}\]
for some constant $c_k$, which does not depend on $T$ or $\omega$ (see \cite{KPV_CPAM} pages 584  and  586).

Proposition \ref{SI_Xk} implies that,
if the operator $\Phi$ is regular enough (that is,  $\Phi \in L^{0,1+\epsilon}_2$ for some positive $\epsilon$ when $k=2$  or $\Phi\in L^{0,\frac{5}{12}}_2$ when $k=3$),
then the random process
$v$,  defined by  $v(t)=\int_0^t S(t-s) \, \Phi \, dW(s)$, belongs a.s. to ${\mathcal X}_k^T$.
Furthermore, the  map $v(.) $ belongs a.s. to $C([0,T];H^{\sigma(k)}_x)$ for any $T>0$. For $k=2,3$ and $R>0$ set
\[ {\mathcal Y}^{R,T}_k:=\big\{ u\in C\big([0,T],H^{\sigma(k)}_x\big) \cap {\mathcal X}_k^T \; : \; \|u\|_{{\mathcal X}_k^T}\leq R\big\}.\]
Let ${\mathcal F}_k$ denote the map defined by
\[ \big( {\mathcal F}_k u\big) (t)=S(t) \, u_0 + v(t) -  \int_0^t S(t-s) \big( u^k \partial_x u\big)(s) ds.\]
Let $k=2$; using inequalities proved in \cite{KPV_CPAM} page 584-585,
 we deduce that for $u_0\in H^{\frac{1}{4}}_x$ a.s. and $\Phi\in L^{0,1+\epsilon}_2$ for some positive $\epsilon$ given $u, u_1, u_2 \in {\mathcal X}_2^T$, we have
\begin{align*}
\|{\mathcal F}_2 u\|_{{\mathcal X}_2^T} &\leq
c_2 \|D^{\frac{1}{4}}_x u_0 \|_{L^2_x} + \| v\|_{{\mathcal X}_2^T} +\tilde{C}_2\, T^{\frac{1}{2}}
\, \|u\|_{{\mathcal X}_2^T}^3,\\
\|{\mathcal F}_2 u_1 - {\mathcal F}_2 u_2\|_{{\mathcal X}_2^T} &\leq \bar{C}_2\, T^{\frac{1}{2}} \big( \|u_1\|_{{\mathcal X}_2^T}^2
+ \|u_2\|_{{\mathcal X}_2^T}^2 \big) \, \|u_1-u_2\|_{{\mathcal X}_2^T}.
\end{align*}
For almost every $\omega$ choose
\begin{equation} \label{defR2}
 R_2(\omega) = 2\Big( c_2 \|u_0(\omega)\|_{H^{\frac{1}{4}}_x} +
\|v(\omega)\|_{{\mathcal X}_2^T}\Big),
\end{equation}
 and let
$T_2(\omega)>0$ satisfy
\begin{equation} \label{defT2}
 2\,  \tilde{C}_2\,  T_2(\omega)^{\frac{1}{2}}\, R_2(\omega)^2 \leq 1
\quad \mbox{\rm and}\quad    4\, \bar{C}_2 T_2(\omega)^{\frac{1}{2}}\, R_2(\omega)^2 \leq 1.
\end{equation}

In a similar way, when $k=3$, the
inequalities proved in \cite{KPV_CPAM} page 590
 imply that for $u_0\in H^{\frac{1}{12}}_x$ a.s. and $\Phi\in L^{0,\frac{5}{12}}_2$, given $u, u_1, u_2
\in {\mathcal X}_3^T$, we have for some  $\rho>0$
\begin{align*}
\|{\mathcal F}_3 u\|_{{\mathcal X}_3^T} &\leq c_3 \|D^{\frac{1}{12}}_x u_0 \|_{L^2_x} + \| v\|_{{\mathcal X}_3^T} +\tilde{C}_3\, T^{\frac{1}{18}}\,
(1+T)^\rho\,   \|u\|_{{\mathcal X}_3^T}^4,\\
\|{\mathcal F}_3 u_1 - {\mathcal F}_3 u_2\|_{{\mathcal X}_3^T} &\leq \bar{C}_3\, T^{\frac{1}{18}}\, (1+T)^\rho\,  \big[ \|u_1\|_{{\mathcal X}_3^T}^3
+ \|u_2\|_{{\mathcal X}_3^T}^3 \big] \, \|u_1-u_2\|_{{\mathcal X}_3^T}.
\end{align*}
For almost every $\omega$ choose
\begin{equation} \label{defR3}
R_3(\omega) = 2\Big( c_3 \|u_0(\omega)\|_{H^{\frac{1}{12}}_x} + \|v(\omega)\|_{{\mathcal X}_3^T}\Big),
\end{equation}
 and let
$T_3(\omega)>0$ be such that
\begin{equation}  \label{defT3}
 2\,  \tilde{C}_3\,  T_3(\omega)^{\frac{1}{18}}\, \big(1+T_3(\omega)\big)^\rho \,  R_3(\omega)^3 \leq 1
\quad \mbox{\rm and}\quad    4\, \bar{C}_3 \, T_3(\omega)^{\frac{1}{18}}\, \big(1+T_3(\omega)\big)^\rho \, R_3(\omega)^2 \leq 1.
\end{equation}
These choices imply that for $k=2,3$, ${\mathcal F}_k$ maps ${\mathcal Y}_k^{R_k(\omega), T_k(\omega)}$ into itself.
Furthermore,  since
$\| {\mathcal F}_k u_1 - {\mathcal F}_k u_2\|_{{\mathcal X}_k^T}\leq \frac{1}{2} \|u_1-u_2\|_{{\mathcal X}_k^T}$ for
$u_1, u_2\in {\mathcal Y}_k^{R_k(\omega), T_k(\omega)}$, the map ${\mathcal F}_k$  is a strict contraction on that set.
 Hence, ${\mathcal F}_k$ has a unique fixed point
in ${\mathcal Y}_k^{R_k(\omega), T_k(\omega)}$, which is the unique solution to \eqref{stoch_gKdV} in ${\mathcal X}_k^{T_k(\omega)}$, $k=2,3$, thus, concluding the proof.
 \end{proof}

 \section{Global well-posedness}\label{Global}
 We now prove global existence when the initial condition $u_0$ is in $H^1_x$ a.s.
 The argument relies on a regularization of $u_0$ and $\Phi$ and on the following conservation laws.
 When $k=2,3$ and $z_k$ is the (deterministic) solution to the gKdV equation
 \[ \partial_t z_k(t) + \big( \partial^3_x z_k(t) + z_k(t)^k \partial_x z_k(t) \big)  =0, \quad z_k(0)=z_0\in H^1_x,\]
 then the following quantities are time-invariant 
 \begin{align}
\mbox{\rm the  mass:}\quad & \|z_k(t)\|_{L^2_x}^2,     \label{mass}\\
\mbox{\rm the  Hamiltonian:} \quad& {\mathcal H}_k(z_k(t))= \frac{1}{2}  \int_{\RR} |D_x z_k(t)|^2 dx
  - \frac{1}{(k+1)(k+2)} \int_{\RR} z_k(t) ^{k+2} dx.    \label{Hamiltonian}
 \end{align}

We now prove Theorem \ref{global}.

 \begin{proof} We  suppose that $u_0\in  L^2_\omega(H^1_x) \cap    L^{2q}_\omega(L^2_x)$
 for some $q\in [2,\infty) $ to be chosen later.

   The proof is based on   approximations   of $\Phi$ and $u_0$ and contains several steps.
   {Indeed, we want to obtain moments of the $H^1_x$-norm of $u_n$ uniformly in $t$. The mild formulation does not allow us to use
   martingale estimates for the stochastic integral appearing when the It\^o formula is applied to the mass and to the Hamiltonian. 
   Thus,  we have to use a sequence of strong solutions $\{u_n\}$ of \eqref{stoch_gKdV}, where $\Phi_n$ is  a ``smoother" Hilbert-Schmidt operator and
 $  u_{0,n}$ is a ``smoother" initial condition.    }
 Let $\Phi_n\in L^{0,4}_2$ and   $u_{0,n}\in H^3_x$ be such that
 \begin{align}
 &\Phi_n \to \Phi \quad \mbox{\rm in } L^{0,1+\epsilon}_2, \; \epsilon >0 \; (\mbox{\rm resp. in } L^{0,1}_2)\;  \mbox{\rm for }
  k=2\;  (\mbox{\rm resp. } k=3),       \label{conv_Phi}\\
& u_{0,n}\to u_0 \quad \mbox{\rm in } L^2_\omega(H^1_x) \cap    L^{2q}_\omega(L^2_x) \quad \mbox{\rm and in } H^1_x\;
 \mbox{\rm a.s.}  \label{conv_initial}
 \end{align}
 \smallskip

{\bf Step 1.}~  Proposition \ref{SI_Xk} proves that the sequence  $v_n(t):=\int_0^t S(t-s) \Phi_n dW(s)$
 converges to the stochastic integral $v$ in $L^2_\omega({\mathcal X}_k^T)$. Hence, there exists a subsequence, still denoted $\{v_n\}$,
 which converges to $v$ a.s. Furthermore, for any integer $n$ and $k=2,3$, there exists a unique solution $u_n$ to
 \[ \partial_t u_n(t) + \big( \partial^3_x u_n(t) + u_n(t)^k \partial_x u_n(t) \big) dt =0, \quad u_n(0)=u_{0,n}, \]
and $u_n$ belongs a.s. to $L^\infty_t(H^3_x)$. Indeed, following the argument in \cite{deB_Deb_sKdV}, Lemma 3.2, if we set
$v_n(t )= \int_0^t S(t-s) \Phi_n dW(s)$ and let $z_n=u_n-v_n$,
then $z_n$ has to solve a.s. the deterministic equation
\[ \partial_t z_n(t) + \big[ \partial^3_x z_n(t) + \big(z_n(t)+v_n(t)\big)^k \partial_x \big( z_n(t)+v_n(t)\big) \big] dt =0, \quad z_n(0)=u_{0,n}. \]
To ease notations we do not specify the value of $k=2,3$ when dealing with the solution $u_n$.
 Standard arguments such as the parabolic regularization described in \cite{Temam}
 yield  that the above equation has  a unique local solution. Finally, an argument similar to that in \cite{Gardner} proves that the invariant quantities in \eqref{mass}
 and \eqref{Hamiltonian}
 allow us to extend this solution to any time interval $[0,T]$.  
 Note that $u_n\in L^\infty_t(H^3_x)\cap {\mathcal X}_k^T$ a.s.
 \smallskip

 {\bf Step 2.} ~
 We next prove that the sequence $(u_n)$ is bounded in $L^{2q}_\omega(L^\infty_t(L^2_x))$. The proof is based on It\^o's formula for
 the mass and conservation of the mass ($L^2_x$-norm) of the solutions to the deterministic gKdV equation.

Using the   conservation of mass  for the solutions   to the deterministic gKdV equation, we get
$$
\int_0^t \big( u_n(s), \partial_x^3 u_n(s)+u_n(s)^k \partial_x u_n(s)\big) ds =0.
$$ 
Note that this requires  $u_n(s)\in H^3_x$ a.s., which holds by Step 1, and
$u_n(s)\in L^{2(k+1)}_x$ a.s., which is true, since $H^1_x\subset L^{2(k+1)}_x$.
 It\^o's formula applied to $\|u_n(t)\|_{L^2_x}^2$  and the identity $\sum_{j\geq 0} \|\Phi_n e_j\|^2_{L^2_x} = \|\Phi\|^2_{L_2^{0,0}}$ yield
 \[ 
 \|u_n(t)\|_{L^2_x}^2= \|u_{0,n}\|_{L^2_x}^2 + 2\int_0^t \big( u_n(s),  \Phi_n dW(s)\big) + t  \|\Phi_n\|_{L^{0,0}_2}^2 .
 \] 
 {Using once more It\^o's formula with the map $y\mapsto y^q$, $q\in [2,\infty)$, and the process $\|u_n(t)\|_{L^2_x}^{2}$,  we obtain }
 \begin{equation}   \label{Ito_mass}
 \|u_n(t)\|_{L^2_x}^{2q}= \|u_{0,n}\|_{L^2_x}^{2q} + 2q\int_0^t \|u_n(t)\|_{L^2_x}^{2(q-1)} \big( u_n(s),  \Phi_n dW(s)\big)
 + R(t) ,
 \end{equation}
 where
 \begin{align*}
  R(t)=&   q \int_0^t \|u_n(s)\|_{L^2_x}^{2(q-1)} \|\Phi_n\|_{L_2^{0,0}}^2 ds
 + 2q(q-1) \int_0^t \|u_n(s)\|_{L^2_x}^{2(q-2)} \sum_{j\in \N} \big( u_n(s), \Phi_n e_j\big)^2 ds  .
 \end{align*}
  The Cauchy-Schwarz inequality applied to the last term gives
 \begin{equation}    \label{Rt}
 |R(t)| \leq  \|\Phi_n\|_{L^{0,0}_2}^2 \int_0^t (2q^2 -q) \|u_n(s)\|_{L^2_x}^{2(q-1)} ds
 \leq  \frac{1}{4} \sup_{s\in [0,T]}  \|u_n(s)\|_{L^2_x}^{2q} + C(T) \|\Phi_n\|_{L^{0,0}_2}^{2q},
 \end{equation}
 for some  $C(T)>0$ which is an increasing function of $T$, where the last inequality is obtained using
 Young's inequality  with the conjugate exponents $q$ and $\frac{q}{q-1}$.
 Furthermore, the Davies inequality for stochastic integrals, the Cauchy-Schwarz
 and then the Young inequality applied with the conjugate  exponents  $ 2q$ and $\frac{2q}{2q-1} $ imply
 \begin{align}    \label{Davies_mass}
 E\Big( &\sup_{t\in [0,T]} \int_0^t \! \!\|u_n(s)\|_{L^2_x}^{2(q-1)}  \big( u_n(s) , \Phi_n dW(s)\Big) \leq
  3 E \Big(\Big\{ \!\int_0^T\! \!\|u_n(s)\|_{L^2_x}^{4(q-1)} \!\sum_{j\geq 0} \big( u_n(s)  ,  \Phi_n e_j\big)^2 ds \! \Big\}^{\frac{1}{2}}\Big)
  \nonumber \\
 &\leq 3 E\Big(  \Big\{ \int_0^T \|u_n(s)\|_{L^2_x}^{4q-2}
 \, \|\Phi_n\|_{L^{0,0}_2}^2 ds\Big\}^{\frac{1}{2}} \Big)
 \; \leq 3 E\Big(   \sup_{s\in [0,T]}  \|u_n(s)\|_{L^2_x}^{2q-1}
 \, \sqrt{T}\,  \|\Phi_n\|_{L^{0,0}_2}  \Big)
 \nonumber  \\
 &\leq \frac{1}{4}  E\Big( \sup_{s\in [0,T]} \|u_n(s)\|_{L^2_x}^{2q}\Big)
 + C(T) \, \|\Phi_n\|_{L^{0,0}_2}^{2q},
 \end{align}
 for some $C(T)>0$, which is an increasing function of $T$.
 The inequalities \eqref{Ito_mass}-\eqref{Davies_mass} yield  the existence of a constant $C(T)>0$ such that
 \begin{equation}    \label{moments_L2norm}
  E\Big(  \sup_{s\in [0,T]} \|u_n(s)\|_{L^2_x}^{2q}\Big) \leq 2 E \big(  \|u_{0,n}\|_{L^2_x}^{2q} \big) + C(T) \|\Phi_n\|_{L^{0,0}_2}^{2q}.
 \end{equation}

{\bf Step 3.} ~ 
We now prove that $(u_n)$ is bounded in $L^2_\omega(L^\infty_t(H^1_x))$.

To upper estimate the $H^1_x$ norm of $u_n$, we use the Hamiltonian   ${\mathcal H}_k$ defined in \eqref{Hamiltonian}.
The time  invariance  of the Hamiltonian,  aka conservation of energy,  for the solution to the
deterministic gKdV equation yields
\[
\int_0^t {\mathcal H}_k'(u_n(s)) \big[ \partial_x^3 u_n(s) + u_n(s)^k \partial_x u_n(s) \big] ds =0,
\]
{ where
for $\varphi, \psi \in H^3_x$, we have
\[ {\mathcal H}_k'(\varphi)(\psi) = \int_{\RR} D_x \varphi D_x\psi dx - \frac{1}{k+1} \int_{\RR} \varphi^{k+1} \psi dx
= - \int_{\RR} \big[ D^2_x \varphi  + \frac{1}{k+1}  \varphi^{k+1} \big] \psi dx.\]
Note that this integral makes sense for $u_n(s)$.  Indeed, the Gagliardo-Nirenberg inequality implies
$H^1_x\subset L^q_x$ for any $q\in [2,\infty)$ and,  since $u_n\in H^3_x$ a.s., we have  $u_n(s)\in  L^p_x$  for any $p\in [2,\infty)$.
Hence,   $u_n(s)^{k+1}  \in L^2_x$ a.s. }

Integration by parts implies that for $\varphi\in H^3_x$, the bilinear form ${\mathcal H}_k''(\varphi)$ can be written as 
\begin{equation} \label{H"}
 {\mathcal H}_k''(\varphi)(v_1,v_2) = \big( \partial_x v_1\, , \, \partial_x v_2\big) -\int_{\RR} \varphi^k \,  v_1\, v_2 dx, \quad v_1, v_2\in H^3_x.
 \end{equation}
 Since $\Phi_n\in L^{0,4}_2$, the vectors $\Phi_n e_j\in H^3_x$.
Thus, the It\^o formula applied to  ${\mathcal H}_k(u_n)$ yields 
\begin{align}     \label{Ito_Hamilton}
{\mathcal H}_k(u_n(t)) =& {\mathcal H}_k(u_0) -
\int_0^t \Big[ \big( \partial_x^2 u_n(s) ,\Phi_n dW(s)\big) + \frac{1}{k+1} \big( u_n(s)^{k+1},\Phi_n dW(s)
\big) \Big]     \nonumber \\
&  + \frac{1}{2} \int_0^t 
\sum_{j\geq 0} {\mathcal H}_k''(u_n(s) )(\Phi_n e_j , \Phi_n e_j) ds.
\end{align}
Using the explicit form of \eqref{H"}, we obtain
 \begin{align*}
 \sum_{j\geq 0} {\mathcal H}_k''(u_n(s) )(\Phi_n e_j , \Phi_n e_j) =
  &\sum_{j\in \N} \int_{\RR} \Big[  | \partial_x (\Phi_n e_j)|^2   -  | u_n(s)|^k \big( \Phi_n e_j\big)^2 \Big] dx \\
   \leq &  \| \Phi_n\|_{L^{0,1}_2}^2 +  \sum_{j\in \N} \int_{\RR}  \| \Phi_n e_j\|_{L_x^\infty}^2
  |u_n(s)|^k    dx  \\
 \leq & \| \Phi_n\|_{L^{0,1}_2}^2 + C  \;  \|\Phi_n\|^2_{L^{0,1}_2}\; \|u_n(s)\|_{L^{k}_x}^k,
 \end{align*}
 where we used the Sobolev embedding $H^1_x\subset L^\infty_x$ to obtain the last upper estimate.

For $k=2$ the last expression simplifies to
\begin{equation}  \label{trace_k=2}
\sum_{j\geq 0} {\mathcal H}_2''(u_n(s) )(\Phi_n e_j , \Phi_n e_j)  \leq \| \Phi_n\|_{L^{0,1}_2}^2 + C  \;  \|\Phi_n\|^2_{L^{0,1}_2}\; \|u_n(s)\|_{L^{2}_x}^2
\end{equation}
for some constant $C >0$.

For $k=3$, the Gagliardo-Niremberg inequality implies $\| u_n\|_{L^3_x} \leq C \|u_n\|_{H_x^1}^\alpha \, \|u_n\|_{L^2_x}^{1-\alpha}$ for
$\alpha = \frac{1}{2} - \frac{1}{3} = \frac{1}{6}$.
Therefore, using Young's inequality with the conjugate exponents $4$ and $4/3$, we get
  \begin{equation}         \label{trace_k=3}
 \sum_{j\geq 0} {\mathcal H}_3''(u_n(s) )(\Phi_n e_j , \Phi_n e_j)  \leq
   \epsilon \|u_n(s)\|_{H^1_x}^2 + C(\epsilon) \,  \| \Phi_n\|_{L^{0,1}_2}^{\frac{8}{3}} \, \|u_n(s)\|_{L^2_x}^{\frac{10}{3}}\,   + \| \Phi_n\|_{L^{0,1}_2}^2,
  \end{equation}
 for any small  constant $\epsilon>0$ to be chosen later, and some positive constant $C(\epsilon)$.

 As in \eqref{Davies_mass}, using once more the Davies inequality for the stochastic integral,
  integration by parts and the Cauchy-Schwarz inequality, we obtain
 \begin{align*}
 E\Big( \sup_{t\in [0,T]} & \int_0^t -  \Big( \partial_x^2 u_n(s) + \frac{1}{k+1}  u_n(s)^{k+1},\Phi_n dW(s)
\Big) \Big) \\
&\leq 3 E\Big( \Big\{ \int_0^T \sum_{j\geq 0}  \Big( \partial_x^2 u_n(s) + \frac{1}{k+1} u_n(s)^{k+1}\, , \, \Phi_n e_j \Big)^2 ds \Big\}^{\frac{1}{2}} \Big)\\
&\leq 3 \sqrt{2} E\Big( \Big\{ \int_0^T \Big[ \sum_{j\geq 0} \big( \partial_x u_n(s)\, , \, \partial_x \Phi_n e_j\big)^2 +
\sum_{j\geq 0} \Big( \frac{1}{k+1}  u_n(s)^{k+1},\Phi_n e_j\Big)^2 \Big] ds \Big\}^{\frac{1}{2}} \Big) \\
& \leq C \sqrt{T} \|\Phi_n\|_{L^{0,1}_2} \Big[ E\Big( \sup_{s\in [0,T]}  \|u_n(s)\|_{H^1_x} \Big) + E\Big( \sup_{x\in [0,T]} \|u_n(s)\|_{L^{k+1}_x}^{k+1} \Big)
\Big],
 \end{align*}
 where the last inequality follows from the Sobolev embedding $H_x^1\subset L^\infty_x$.
 
 The Gagliardo-Nirenberg inequality implies $\|u_n\|_{L^{k+1}_x} \leq \| u_n\|_{H^1_x}^\beta \, \|u_n\|_{L^2_x}^{1-\beta}$,
 where $\beta = \frac{1}{2}-\frac{1}{k+1}= \frac{k-1}{2(k+1)}$.
Using H\"older's and  Young's inequalities with the conjugate exponents $\frac{4}{k-1}$ and $\frac{4}{5-k}$,  we obtain
 \begin{align}   \label{Davies_Hamilton}
 E\Big(& \sup_{t\in [0,T]}    \int_0^t - \Big[ \big( \partial_x^2 u_n,\Phi_n dW(s)\big) + \frac{1}{k+1} \big( u_n(s)^{k+1},\Phi_n dW(s)
\big) \Big] \Big)    \nonumber \\
\leq & \frac{\epsilon}{2} E\Big( \sup_{s\in [0,T]} \|u_n(s)\|_{H^1_x}^2 \Big) + C(\epsilon)\, T\,  \|\Phi_n\|_{L^{0,1}_2}^2   \nonumber \\
&
+ C \sqrt{T} \|\Phi_n \|_{L^{0,1}_2} E\Big( \sup_{s\in [0,T] } \|u_n(s)\|_{H^1_x}^{\frac{k-1}{2}}
\; \sup_{s\in [0,T]} \|u_n(s)\|_{L^2_x}^{\frac{k+3}{2}} \Big)  \nonumber \\
\leq & \epsilon E\Big( \sup_{s\in [0,T]} \|u_n(s)\|_{H^1_x}^2 \Big) + C(\epsilon)\, T\,  \|\Phi_n\|_{L^{0,1}_2}^2
+  C(\epsilon, T) \, \|\Phi_n\|_{L^{0,1}_2}^{\frac{4}{5-k}}\,   E\Big(\sup_{s\in [0,T]} \|u_n(s)\|_{L^2_x}^{\frac{2(k+3)}{5-k}} \Big) \,
\end{align}
for some number $C(\epsilon, T)>0$, which is again an increasing function of $T$ for fixed $\epsilon>0$.
 Note that for $k=2$, $\frac{k+3}{5-k}=\frac{5}{3} <2 $, and for $k=3$ we have $\frac{k+3}{5-k}=3$.

 Collecting the information from the estimates
  \eqref{Ito_Hamilton}-\eqref{Davies_Hamilton} and choosing $\epsilon=\frac{1}{16}$,
    we obtain for $q(2)=2$ (resp. $q(3)=3$) the existence of a positive constant $C(T)$ such that
  \begin{align*}
  E\Big( \sup_{t\in [0,T]}  {\mathcal H}_k(u_n(t))\Big) \leq & E\big( {\mathcal H}_k(u_{0,n}) \big)
  + \frac{1}{8} E\Big( \sup_{t\in [0,T]} \|u_n(s)\|_{H^1_x}^2\Big)+ C\, T\, \|\Phi_n\|_{L^{0,1}_2}^2  \\
  &  +C(T) \big( 1+ \|\Phi_n\|_{L^{0,1}_2}^{\frac{8}{3}}\big) \Big[ 1+ E\Big(\sup_{s\in [0,T]} \|u_n(s)\|_{L^2_x}^{2 q(k)} \Big)\Big] .
  \end{align*}
  Finally, the Gagliardo-Nirenberg inequality implies that $\|\varphi\|_{L^{k+2}_x} \leq C \| \varphi \|_{H^1_x}^\gamma \, \|\varphi\ |_{L^2_x}^{1-\gamma}$
  for $\gamma = \frac{1}{2}-\frac{1}{k+2} = \frac{k}{2(k+2)}$. Thus, using  Young's inequality with the conjugate exponents
  $\frac{4}{k}$ and $\frac{4}{4-k}$, we deduce
  \[ \frac{1}{4} \|u_n(s)\|_{H^1_x}^2 - C \|u_n(s)\|_{L^2_x}^{\frac{2(k+4)}{4-k}} \leq {\mathcal H}_k(u_n(s))
  \leq \frac{3}{4} \|u_n(s)\|_{H^1_x}^2 +C \|u_n(s)\|_{L^2_x}^{\frac{2(k+4)}{4-k}}
  \]
  for some constant $C>0$.
Let $\tilde{q}(k)=\frac{(k+4)}{4-k}$; then $\tilde{q}(2)=3>q(2)$, $\tilde{q}(3)=7>q(3)$.  For
$u_0\in L^{2\tilde{q}(k)}_\omega(L^2_x)$ we have for some positive constant $C(T)$
 \begin{align*}
 \frac{1}{4}  E\Big(&  \sup_{t\in [0,T]} \|u_n(s)\|_{H^1_x}^2\Big) \leq  \frac{1}{8} E\Big( \sup_{t\in [0,T]} \|u_n(s)\|_{H^1_x}^2\Big)
 + \frac{3}{4} E\big( \|u_{0,n}\|_{H^1_x}^2\big)  +
 C E\big(\|u_{0,n}\|_{L^2_x}^{2\tilde{q}(k)} \big)  \\
& \quad 
+ C(T)  \big( 1+\|\Phi_n\|_{L^{0,1}_2}^{\frac{8}{3}} \big)
 \Big[ 1+ E\Big(\sup_{s\in [0,T]} \|u_n(s)\|_{L^2_x}^{2 q(k)} \Big) \Big]   + C E\Big(\sup_{s\in [0,T]} \|u_n(s)\|_{L^2_x}^{2 \tilde{q}(k)} \Big).
 \end{align*}
 Furthermore,   if  $u_0\in L^{2\tilde{q}(k)}_\omega(L^2_x)$, choosing  the exponent
 $q=\tilde{q}(k)\geq 2 $ used for the approximation $u_{0,n}$ of $u_0$, we deduce from \eqref{moments_L2norm}  that
 $\|u_n\|_{L^2_\omega(L^\infty_t(H^1_x))}$ is bounded in terms of
  $\| \Phi_n\|_{L^{0,1}_2}$
and $\|u_{0,n}\|_{L^{2\tilde{q}(k)}_\omega(L^2_x)}$.  Since these norms are bounded by a constant independent of $n$,
by virtue of the convergence we have required in Step 1, we can now deduce that the sequence  $\{u_n\}$
 is bounded in $L^2_\omega(L^\infty_t(H^1_x))$.
 \smallskip

 {\bf Step 4.}~ The bound of $\{u_n\}$ proved in Step 3 implies the existence of a random variable
 $\tilde{u}\in L^2_\omega(L^\infty_t(H^1_x))$ and of a subsequence (still denoted $\{u_n\}$) such that
 \[ u_n \rightharpoonup  \tilde{u} \quad \mbox{\rm in } L^2_\omega(L^\infty_t(H^1_x)) \; \mbox{\rm weak star}.
 \]
  Technically speaking, $\tilde{u}\in L_{\omega,w^*}(L^\infty_t(H^1_x))$,  since we have used the weak star limit.
 Nevertheless,  $\tilde{u}\in L^\infty_t(H^1_x)$ a.s.
 Recall  $R_2(\omega)$ and $R_3(\omega)$ from \eqref{defR2} and \eqref{defR3}, respectively.
Let $\tilde{R}_k(\omega)$ be defined by
 \[
  \tilde{R}_k(\omega):=2\big[ c_k (\|u_0(\omega)\|_{H^1_x} + \| \tilde{u}(\omega)\|_{L^\infty_t(H^1_x)}\big) + \| v(\omega)\|_{{\mathcal X}_k^T} \big]
  \geq R_k(\omega), \quad k=2,3,
  \]
  where $v(t)=\int_0^t S(t-s) \Phi dW(s)$.
  Next recall $T_2(\omega)$ and $T_3(\omega)$ from \eqref{defT2} and \eqref{defT3}, respectively.
  Choose $\tilde{T}_k(\omega)>0$, $k=2,3$,  such that
  inequalities similar to  \eqref{defT2} and \eqref{defT3} are satisfied with $\tilde{T}_k(\omega)$ and $\tilde{R}_k(\omega)$ instead of $T_k(\omega)$ and $R_k(\omega)$, respectively.
  Note that $\tilde{T}_k(\omega) \in (0, T_k(\omega)]$.
 Let ${\mathcal F}_{n,k}$, $k=2,3$, $n\in \N$ be defined on ${\mathcal X}_k^{\tilde{T}_k(\omega)}$ by
  \[ ({\mathcal F}_{k,n} z) (t) = S(t) u_{0,n} + v_n (t)+ \int_0^t S(t-s) z(s)^k \partial_x z(s) ds,\]
  where $v_n(t)=\int_0^t S(t-s) \Phi_n  dW(s)$.

  {
  From Step 1 we know that  a.s. $u_n(\omega)\in {\mathcal X}_k^{\tilde{T}_k(\omega)}$,  and that  a.s.
  $u_n(\omega)$ is the unique fixed point of the map ${\mathcal F}_{k,n}$ on the ball of radius $\tilde{R}_k(\omega)$
  of ${\mathcal X}_k^{\tilde{T}_k(\omega)}$. Indeed, on that
  ball ${\mathcal F}_{k,n}$ is  a contraction,
  since by construction we know that $\| {\mathcal F}_{k,n} z_1- {\mathcal F}_{k,n} z_2\|_{{\mathcal X}_k^{\tilde{T}_k(\omega)}}
  \leq \frac{1}{2} \|z_1-z_2\|_{{\mathcal X}_k^{\tilde{T}_k(\omega)}}$.
  The convergences from \eqref{conv_Phi} and
  \eqref{conv_initial} prove that
 $\|S(t) u_0 - S(t) u_{0,n}\|_{{\mathcal X}_k^T} $ and $\|v - v_n\|_{{\mathcal X}_k^T} $ converge to 0 as $n\to \infty$  for every $T>0$.
    Furthermore, we have
  \[ \| {\mathcal F}_{k,n} u_n - {\mathcal F}_{k} u\|_{{\mathcal X}_k^{\tilde{T}_k(\omega)}} \leq \| u_{0,n}-u_0\|_{{\mathcal X}_k^{\tilde{T}_k(\omega)}}
  + \| v_n-v \|_{{\mathcal X}_k^{\tilde{T}_k(\omega)}} + \frac{1}{2}  \|  u_n -  u\|_{{\mathcal X}_k^{\tilde{T}_k(\omega)}}.
  \]
 Hence, $u_n$ converges to $u$ a.s. in ${\mathcal X}_k^{\tilde{T}_k(\omega)}$, where $u$ is the unique fixed
  point of ${\mathcal F}_k$ on the ball of radius $\tilde{R}_k(\omega)$ of ${\mathcal X}_k^{\tilde{T}_k(\omega)}$.

   This implies that $u(\omega)=\tilde{u}(\omega)$ a.s. on the time interval $[0,\tilde{T}_k(\omega)]$. Since $\tilde{u}\in L^\infty_t(H^1_x)$ a.s.,
   given $\alpha \in (0,1)$, we may choose  $\tau_k(\omega) \in [\alpha \tilde{T}_k(\omega), \tilde{T}_k(\omega)]$ such  that
  $\| u(\tau_k(\omega))\|_{H^1_x} \leq \| \tilde{u}\|_{L^\infty_t(H^1_x)}$.
  Replacing the initial condition $u_0$ by
  $u(\tau_k(\omega))$, this enables us to define a solution on the time interval
  $\big[\tau_k(\omega), \big(\tau_k(\omega) + \tilde{T}_k(\omega)\big)\wedge T\big]$.
  Thus, we can  inductively define  a solution on any fixed time interval $[0,T]$ a.s.
   Indeed,  $\tilde{T}_k(\omega)>0$ a.s. and at each step we increase the
  length of the time interval by at least $\alpha \tilde{T}_k(\omega)$.

 Finally, as in \cite{KPV_CPAM} we obtain that $S(t) u_0$ is a.s. continuous from $[0,T]$ to $H^1_x$. The stochastic integral
 $v(t)=S(t) \int_0^t S(-s) \Phi dW(s)$ also belongs to $C([0,T], H^1_x)$ a.s.  Hence, as in the deterministic framework of \cite{KPV_CPAM},
 we deduce that $u\in C([0,T],H^1_x)$ a.s. This concludes the proof.  }
 \end{proof}

 \noindent {\bf Acknowledgments:} This work started when both authors participated in the semester program ``New Challenges in PDE :
 Deterministic dynamics and randomness in high and infinite dimensional systems"  at MSRI in Fall 2015.
 They would like to thank MSRI for the financial support and the excellent working conditions.
 The project continued when the second author participated in the special trimester ``Nonlinear wave equations" at the IHES in Summer 2016 and that
 excellent working environment gave an additional boost to this collaboration, for which both authors are very thankful. S.R. was partially supported by the NSF CAREER grant \# 1151618.

\bibliographystyle{amsplain}

\end{document}